%% file: main.tex
\input{_0_vars}

\ifdefined\isarxiv
\documentclass[11pt]{article}
\usepackage[numbers]{natbib}
\else
\documentclass{article}
\usepackage{iclr2026_conference,times}

\fi

\ifdefined\isarxiv

\usepackage{amsmath}
\usepackage{amsthm}
\usepackage{amssymb}
\usepackage{algorithm}
\usepackage{subfig}
\usepackage{algpseudocode}
\usepackage{graphicx}
\usepackage{grffile}
\usepackage{wrapfig,epsfig}
\usepackage{url}
\usepackage{xcolor}
\usepackage{epstopdf}

\usepackage{bbm}
\usepackage{dsfont}

\else 

\usepackage{hyperref}
\usepackage{url}

\fi

\ifdefined\isarxiv

\else

\usepackage{amsmath}
\usepackage{amsthm}
\usepackage{amssymb}
\usepackage{algorithm}
\usepackage{subfig}
\usepackage{algpseudocode}
\usepackage{graphicx}
\usepackage{grffile}
\usepackage{wrapfig,epsfig}
\usepackage{url}
\usepackage{xcolor}
\usepackage{epstopdf}

\usepackage{bbm}
\usepackage{dsfont}

\fi

\usepackage{mathrsfs}
\usepackage{tabularx}
\usepackage{longtable}
\usepackage{tikz}
\usepackage[mode=tex]{standalone}
\usetikzlibrary{arrows,arrows.meta,calc,positioning}
 
\ifdefined\isarxiv

\usepackage{hyperref}  
\hypersetup{colorlinks=true,citecolor=blue,linkcolor=blue} 
\usepackage[margin=1in]{geometry}

\else
\fi
 
\graphicspath{{./figs/}}

\theoremstyle{plain}
\newtheorem{theorem}{Theorem}[section]
\newtheorem{lemma}[theorem]{Lemma}
\newtheorem{definition}[theorem]{Definition}

\newtheorem{corollary}[theorem]{Corollary}

\ifdefined\isarxiv

\else
\renewcommand\cite\citep
\fi

\input{_1_maths}

\begin{document}

\ifdefined\isarxiv

\date{September 6, 2026}
\title{\paperTitle}
\author{\paperAuthor}

\else

\title{\paperTitle}

\newcommand{\fix}{\marginpar{FIX}}
\newcommand{\new}{\marginpar{NEW}}

\maketitle

\fi

\ifdefined\isarxiv
  \maketitle
  \begin{abstract}
    \input{00_abstract}

  \end{abstract}

\else

\begin{abstract}
\input{00_abstract}
\end{abstract}

\fi

\input{_2_body}

\ifdefined\isarxiv
\bibliographystyle{alpha}
\bibliography{ref}
\else
\bibliographystyle{iclr2026_conference}
\bibliography{ref}
\fi

\input{_3_app}

\end{document}

%% file: _0_vars.tex
\def\isarxiv{1}

\def\paperTitle{An Improvement to the Upper Bound for Marton's Covering Conjecture}

\def\paperAuthor{
Zhao Song\thanks{\texttt{magic.linuxkde@gmail.com}.} 
\and
Song Yue
}

%% file: _1_maths.tex
\DeclareMathOperator*{\E}{{\mathbb{E}}}

\newcommand{\dist}{\mathrm{d}}

\newcommand{\F}{\mathbb{F}}
\newcommand{\HH}{\mathrm{H}}
\newcommand{\II}{\mathrm{I}}
\newcommand{\KL}{D_{\mathrm{KL}}}
\newcommand{\cU}{\mathrm{U}}
\newcommand{\tauzero}{\tau_0}

\usepackage{booktabs}

%% file: 00_abstract.tex
Marton's covering conjecture concerns sets $A\subseteq\F_2^n$ with
$|A+A|\leq K|A|$ and asks for a cover by polynomially many translates of
one subspace of size at most $|A|$.
Gowers, Green, Manners, and Tao~\cite{ggmt25} proved the conjecture with
covering exponent $12$, and Liao~\cite{l24} improved it to $9$.
We improve the covering exponent from $9$ to $5.287$.

%% file: _2_body.tex
\input{01_intro}

\input{02_tech}
\input{06_preli}

\input{31_upper}
\input{90_llm_disclaimer}

%% file: 01_intro.tex
\section{Introduction}
\label{sec:intro}

For a finite subset $A$ of an abelian group, the doubling constant
$|A+A|/|A|$ measures its additive structure.
Freiman-type theorems seek to describe sets with small doubling in terms of
algebraic objects; over $\F_2^n$, the natural objects are linear subspaces and
their cosets.

Ruzsa~\cite{r99} recorded a sharp covering formulation in 1999 and attributed it to
Marton.
The conjecture asserts that whenever $|A+A|\leq K|A|$, the set $A$ can be
covered by at most $2K^C$ cosets of a single subspace whose cardinality is at
most $|A|$, for an absolute constant $C$.
The central difficulty is to obtain polynomial dependence on $K$, since the
classical Freiman--Ruzsa argument gives a bound exponential in $K$.

Gowers, Green, Manners, and Tao~\cite{ggmt25} proved the covering statement with exponent $C=12$ using an entropic descent argument.
Liao~\cite{l24} subsequently refined the same framework in 2024, introducing a sharper
admissible-penalty analysis and improving the covering exponent from $12$ to
$9$.
Green and Tao~\cite{gt09} proved the lower bound $C\geq 1.466$.

It is therefore natural to ask: 

\begin{center}
  {\it What is the optimal exponent $C$ in Marton's
covering problem?}  
\end{center}

In this work, we improve the upper bound.
\subsection{Our Result}
\label{sec:our:results}

We state our result as follows.

\begin{theorem}[Main upper bound, informal version of
Theorem~\ref{thm:main:formal}]\label{thm:main}
Suppose $n$ is a positive integer, $A\subseteq \F_2^n$ is nonempty, and
$K\geq1$ satisfies $|A+A|\leq K|A|$.
Then there is a subspace $H\leq \F_2^n$ with $|H|\leq|A|$ such that $A$ is
covered by at most $2K^{5.287}$ translates of $H$.
\end{theorem}

Gowers, Green, Manners, and Tao~\cite{ggmt25} obtained the covering exponent
$12$.
Liao~\cite[Theorem~2]{l24} improved it to $9$.
We further improve it to
$5.287$.
Green and Tao~\cite{gt09} give the lower bound $C_*\geq1.466$ for
the optimal exponent $C_*$ in Marton's covering problem.

%% file: 02_tech.tex
\section{Technique Overview}
\label{sec:tech}

Section~\ref{sec:tech:ggmt} presents the Gowers--Green--Manners--Tao argument.
Section~\ref{sec:tech:liao} explains Liao's refinement.
Section~\ref{sec:tech:upper} outlines the proof of our upper bound.

\subsection{Summary of GGMT25}
\label{sec:tech:ggmt}

\paragraph{The shared descent setup.}
For $G$-valued random variables $X$ and $Y$, with independent copies $X'$ and
$Y'$, define their entropic Ruzsa distance by
$\dist[X;Y]:=\HH[X'-Y']-\tfrac12\HH[X]-\tfrac12\HH[Y]$.
It is nonnegative and vanishes precisely when $X$ and $Y$ are uniform on
cosets of one common subgroup~\cite{t10}.
Consequently, if $\cU_A$ is uniform on a set $A$ of doubling at most $K$,
then $\dist[\cU_A;\cU_A]\leq\log K$.

Take four independent variables $X_1,X_2$ distributed as $X$ and $Y_1,Y_2$
distributed as $Y$.
Set $T:=X_1+Y_1$, $\overline T:=X_2+Y_2$, $W:=X_1+X_2$, and
$\overline W:=Y_1+Y_2$.
Also set $S:=T+\overline T$, $V:=X_1+Y_2$, and $d:=\dist[X;Y]$.
These variables satisfy the two fibring identities of
GGMT~\cite[Corollary~4.2]{ggmt25}, which express each of $2d-I_1$ and
$2d-I_2$ as the sum of a sum-side distance and a fibre-side distance, where
$I_1=\II[T:V\mid S]$ and $I_2=\II[T:W\mid S]$.
Testing the four resulting pairs against the descent potential gives, whenever
none of them decreases the potential,

\begin{samepage}
\begin{equation}\label{eq:tech:shared:u}
u:=2\eta d-I_1\geq0,\qquad I_2-2\eta d\leq\frac{\eta}{1-\eta}u.\end{equation}
The constraints in Eq.~\eqref{eq:tech:shared:u}, the
characteristic-two identity $T+V+\overline W=0$, and the
entropic Balog--Szemer\'edi--Gowers estimate of Gowers, Green, Manners, and Tao~\cite[Lemma~A.2]{ggmt25} underlie the original endgame comparison.
\par
\end{samepage}

\paragraph{How GGMT obtain Eq.~\eqref{eq:tech:ggmt-threshold}.}
GGMT use penalties given by Ruzsa distance to two fixed reference variables
and minimise distance plus $\eta$ times the total penalty.
For the three conditioned endgame pairs coming from $(T,V,\overline W)$, let
$\widetilde\delta$ be the sum of their three mutual informations, and let
$P_{\mathrm G}$ be the total increase of the six reference-distance
penalties.

\begin{samepage}
We have
\begin{equation}\label{eq:tech:ggmt:bounds}\widetilde\delta\leq 6\eta d-\frac{1-5\eta}{1-\eta}u,\qquad P_{\mathrm G}\leq(6-3\eta)d+3u.\end{equation}
The first bound in Eq.~\eqref{eq:tech:ggmt:bounds} follows from the first estimate in Section~7 of~\cite{ggmt25}. The second bound follows from the second estimate there.
\par
\end{samepage}

\begin{samepage}
We have
\begin{equation}\label{eq:tech:ggmt:endgame}
d\leq
(1+\frac{\eta}{3})\widetilde\delta
 +\frac{\eta}{3}P_{\mathrm G}.
\end{equation}
where the step follows from the symmetrised endgame lemma of GGMT.
\par
\end{samepage}

\begin{samepage}
We have
\begin{align}
d
&\leq(8\eta+\eta^2)d
 -(
 (1+\frac{\eta}{3})\frac{1-5\eta}{1-\eta}-\eta
 )u\leq(8\eta+\eta^2)d.
\label{eq:tech:ggmt-threshold}
\end{align}
The first step follows from Eqs.~\eqref{eq:tech:ggmt:bounds} and~\eqref{eq:tech:ggmt:endgame}, and simple algebra.
The second step follows from Eq.~\eqref{eq:tech:shared:u}, since
$u\geq0$ and the coefficient of $u$ in the first line is nonpositive in the
relevant range.
For $d>0$ this is contradictory when $8\eta+\eta^2<1$.
GGMT take $\eta=1/9$; their entropic coefficient is then
$\eta^{-1}+2=11$, and the standard conversion from the entropic statement to
a covering statement costs one further power of $K$, giving exponent $12$.
\par
\end{samepage}

\subsection{Summary of Liao24}
\label{sec:tech:liao}

\paragraph{What Liao reuses from GGMT.}
Liao keeps unchanged the entropic Ruzsa distance and its equality case, the
compactness/minimisation descent, the four independent copies and the sums
$T,\overline T,V,W,\overline W,S$, the two fibring identities
\cite[Eqs.~(1) and~(2)]{l24}, the same four primary candidates, and therefore
the same two constraints on $u,I_1,I_2$ in Eq.~\eqref{eq:tech:shared:u}.
He also keeps the identity $T+V+\overline W=0$, the same three conditioned endgame
candidates, and the GGMT entropic Balog--Szemer\'edi--Gowers bound
\cite[Eq.~(4)]{l24}, which bounds the sum of their three distances by
$3I_1+6I_2$.
Thus Liao does not change the distance side of the descent.

\paragraph{What Liao changes.}
Liao abstracts the reference-distance penalties into two admissible functions
$\tau_A,\tau_B$ and uses

\begin{samepage}
\begin{equation}\label{eq:tech:potential}
\phi[X;Y]:=d+\eta\tauzero,
\qquad
\tauzero:=\tau_A(X)+\tau_B(Y).
\end{equation}
\par
\end{samepage}
Admissible penalties are continuous and translation invariant. Their
increase under convolution or conditioning is controlled by an entropy
difference~\cite[Lemma~7]{l24}.
Liao bounds each endgame penalty using sums~\cite[Eqs.~(12)--(14)]{l24}
and fibres~\cite[Eqs.~(15)--(17)]{l24}.
He averages these two estimates for each candidate and sums over the three
candidates to obtain his Eq.~(18):
\begin{samepage}
\begin{equation}\label{eq:tech:liao:penalty}
\tau_{\mathrm{eg}}
 \leq 3\tauzero+6d+(I_2-I_1).
\end{equation}
\par
\end{samepage}
Figure~\ref{fig:tech:liao:endgame} illustrates
Eq.~\eqref{eq:tech:liao:penalty}.
If none of the three endgame candidates decreases $\phi$, their total
potential is at least $3\phi[X;Y]=3d+3\eta\tauzero$.

\begin{samepage}
We have
\begin{equation}\label{eq:tech:liao-threshold}
d\leq I_1+2I_2+\frac{\eta}{3}(\tau_{\mathrm{eg}}-3\tauzero)
\leq I_1+2I_2+\frac{\eta}{3}(6d+I_2-I_1)
\leq 8\eta d+\frac{-3+10\eta}{3(1-\eta)}u
\leq8\eta d.
\end{equation}
The first step follows from the absence of descent,
Eq.~\eqref{eq:tech:potential}, and the GGMT distance bound
\cite[Eq.~(4)]{l24}, followed by simple algebra.
The second step follows from Eq.~\eqref{eq:tech:liao:penalty}.
The third step follows from Eq.~\eqref{eq:tech:shared:u} and simple algebra.
The last step follows because $(-3+10\eta)/(3(1-\eta))<0$ for
$\eta<1/8$ and $u\geq0$ by Eq.~\eqref{eq:tech:shared:u}.
This is the precise endgame improvement: Liao removes the $\eta^2d$ loss from
GGMT's Eq.~\eqref{eq:tech:ggmt-threshold}.
Hence Eq.~\eqref{eq:tech:liao-threshold} is impossible for $d>0$ and
$\eta<1/8$.
\par
\end{samepage}

For the covering statement, Liao makes one further change that is separate
from the endgame calculation.
He takes
$\tau_A=\tau_B=\tfrac12(\tau^-+\tau^+)$, where $\tau^-$ is the least
Kullback--Leibler divergence to a smoothing of $\cU_A$ and
$\tau^+(Z)=\tau^-(Z)+\HH[Z]-\HH[\cU_A]$.
Its exact dense-coset identity~\cite[Lemma~11 and Claim~12]{l24} converts the
coefficient $8$ into a coset of relative density $K^{-8}$.
Ruzsa covering costs one more power of $K$, giving exponent $9$.

\subsection{Summary of Our Upper Bound}
\label{sec:tech:upper}

\paragraph{Diagonal and branch hypotheses.}
We use one common admissible penalty $\tau$ with
$|\tau(X\mid D)-\tau(X)|\leq\frac12\II[X:D]$.
Lemma~\ref{lem:KL:conditioning} proves this property for Liao's KL penalty.
Set $M_\tau:=\min_{V\leq G}2\tau(\cU_V)$.
The assertion $\mathsf D_\tau(B)$ bounds $M_\tau$ by
$2\tau(X)+B\dist[X;X]$ for every law $X$.
For a cyclic law $A+B+C=0$, set
$x:=\II[A:B]$ and $s:=\II[A:C]+\II[B:C]$.
The 84 assertions $\mathsf C_{j,\tau}(1)$ are
\begin{equation}\label{eq:tech:weighted:potential}
M_\tau-2\tau(A\mid C)\leq\alpha_jx+\beta_js,
\qquad j\in\{0,\ldots,83\},
\end{equation}
where Definition~\ref{def:finite:diagonal} selects the exact coefficients.
The assertion $\mathsf C_{j,\tau}(q)$ multiplies its right side by $q$.

\paragraph{Diagonal bounds for posterior ensembles.}
An ensemble consists of conditional laws together with their probabilities.
Define the chain penalty
$\tau^{[N]}(Z):=\sum_{i=1}^N\tau(Z_i\mid Z_1,\ldots,Z_{i-1})$ on $G^N$.
Lemma~\ref{lem:penalty:chain} proves that it belongs to the same penalty
class, that $M_{\tau^{[N]}}=NM_\tau$, and that iterating powers multiplies
their indices.
For a finite ensemble $(p_i,P_i)_{i=1}^m$, if
$\mathsf D_{\tau^{[N]}}(B)$ holds for every $N$, then
Lemma~\ref{lem:inverse:ensemble} gives
\begin{equation}\label{eq:tech:ensemble}
M_\tau\leq2\sum_{i=1}^m p_i\tau(P_i)
 +B\sum_{i=1}^m\sum_{j=1}^m p_ip_j\dist[P_i;P_j].
\end{equation}
Its proof samples one ensemble index and then $N$ independent coordinates
from that law. After division by $N$, the entropy cost of the index tends
to zero. We impose all 85 hypotheses on every
$\rho_N:=\tau^{[N]}$.
The same bound applies to a weighted union of two posterior families.
Choosing the first family with probability $\omega$ gives weights
$\omega^2$ and $(1-\omega)^2$ on the two self-distance terms,
and $2\omega(1-\omega)$ on the cross term.
The complete combination is one inverse cost.

\paragraph{Simultaneous inner and outer comparison.}
The cyclic triple in Section~\ref{sec:tech:ggmt} motivates the inner
construction. After fixing $C$, sample fresh independent copies of
$A\mid C$ and form conditional states.
A rank-two quotient supplies a cyclic law to which
Eq.~\eqref{eq:tech:weighted:potential} applies.
The configurations of Definition~\ref{def:inner:quotient} also permit
quotients retaining the original $C$.
The conditioning rules underlying Eq.~\eqref{eq:tech:liao:penalty},
together with reverse conditioning, control the remaining penalties.
Lemma~\ref{lem:ensemble:flags} identifies the matching posterior ensembles
to which Eq.~\eqref{eq:tech:ensemble} applies.
The outer comparison uses the same rules on independent copies of $X$.

The main diagonal coefficient $B$ in Definition~\ref{def:diagonal:constants}
satisfies $B<4.287$.
It supplies every weaker diagonal cap in $[B,B_0]$.
The 84 cyclic comparisons use at most 32 fresh conditional copies;
the outer comparison uses at most 36 original copies.
Their source pairs are above convex combinations of the input pairs.
Lemma~\ref{lem:cyclic:convex} supplies these source bounds, and
Lemma~\ref{lem:cyclic:substitution} applies them to the conditional
quotients. Each comparison is a nonnegative rational combination of
entropy, conditioning, inverse, and quotient rows, with total cost
mass one. Appendix~\ref{app:coefficients} defines the complete finite
row families and selects their coefficients by rational linear systems.
Lemma~\ref{lem:diagonal:comparison} gives the following implication,
with both sides required for every $N$:
\begin{equation}\label{eq:tech:our-threshold}
\mathsf D_{\rho_N}(B),\ \{\mathsf C_{j,\rho_N}(1)\}_{j=0}^{83}
\quad\Longrightarrow\quad
\mathsf D_{\rho_N}(c),\ \{\mathsf C_{j,\rho_N}(\vartheta)\}_{j=0}^{83},
\end{equation}
where $c<B<4.287$ and $c/B\leq\vartheta:=1-10^{-13}<1$.
Lemma~\ref{lem:finite:feasibility} establishes feasibility of the finite systems.

\paragraph{Uniform initialization and scaling.}
Liao's inverse lemma~\cite[Lemma~7]{l24}, whose closing inequality is
Eq.~\eqref{eq:tech:liao-threshold}, gives $\mathsf D_{\rho_N}(8)$
on every chain power.
Lemma~\ref{lem:branch:initial} gives
$M_{\rho_N}-2\rho_N(A\mid C)\leq8(x+s)$.
Since $8\leq8B,8\alpha_j,8\beta_j$, these estimates give all hypotheses
on the left side of Eq.~\eqref{eq:tech:our-threshold} for $\tau/8$.
If they hold for $\tau/\lambda$, the simultaneous improvement gives them
for $\tau/\max\{1,\vartheta\lambda\}$.
Starting at $\lambda=8$, this iteration reaches one in finitely many steps.
A final comparison at scale one gives the diagonal coefficient $c$ in
Lemma~\ref{lem:diagonal:refined}.

Lemma~\ref{lem:KL:covering} gives the subgroup formula for Liao's KL
penalty from Section~\ref{sec:tech:liao}. It converts the diagonal
bound into a dense coset and yields
$C_0=1+c<5.287$ in Theorem~\ref{thm:main:formal}.
Figure~\ref{fig:tech:ours:endgame} shows the simultaneous comparison
and scaling argument.
Appendix~\ref{app:two:law} uses the diagonal and cyclic bounds to refine
the general comparison. Lemma~\ref{lem:aux:general} gives the
general inverse coefficient $4.793$, and
Corollary~\ref{cor:entropic:general} gives the two-law entropic coefficient
$7.793$. Definition~\ref{def:finite:auxiliary} specifies its finite
coefficient construction.

\begin{figure}[p]
  \centering
  \includegraphics[width=0.98\textwidth]{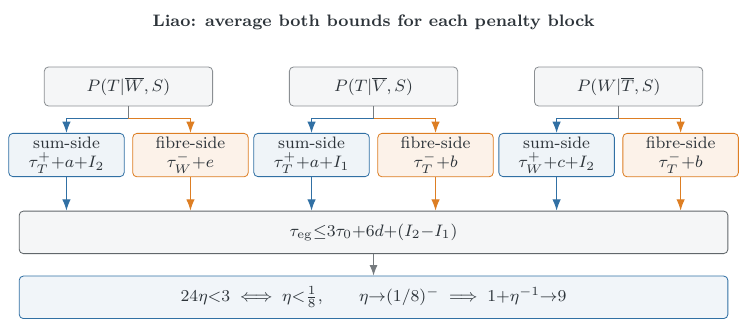}
  \caption{Liao's endgame aggregation~\cite{l24}.
  The two penalty bounds for each branch are averaged to obtain
  Eq.~\eqref{eq:tech:liao:penalty}.}
  \label{fig:tech:liao:endgame}
  \vspace{1.2em}
  \includegraphics[width=0.98\textwidth]{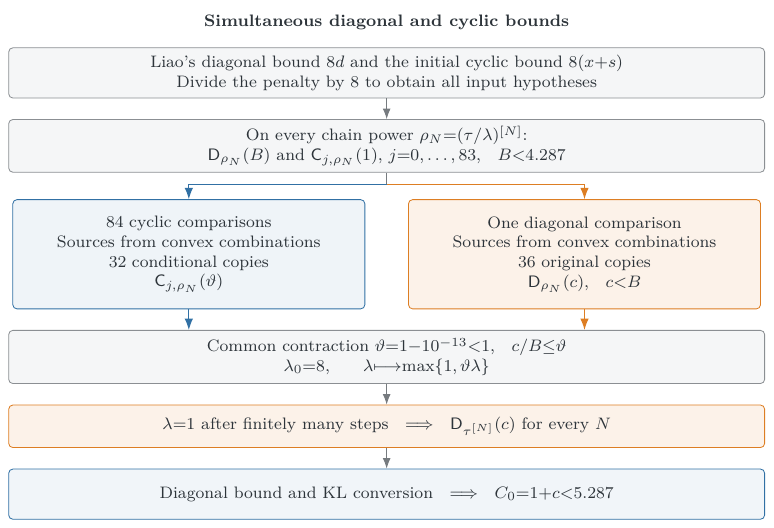}
  \caption{Simultaneous improvement of the diagonal and cyclic bounds
  on every chain power. Convex combinations supply every
  source cyclic inequality. Liao's inverse bound and the
  initial cyclic estimate supply all hypotheses for $\tau/8$.
  Each input bound improves by a factor at most $\vartheta<1$.
  Replacing $\lambda$ by $\max\{1,\vartheta\lambda\}$ reaches scale one
  and gives the covering exponent $C_0<5.287$.}
  \label{fig:tech:ours:endgame}
\end{figure}

%% file: 06_preli.tex
\section{Preliminaries}
\label{sec:preli}

Section~\ref{sec:preli:notation} fixes notation.
Section~\ref{sec:preli:covering} states Marton's covering problem.
Section~\ref{sec:preli:entropy} introduces the distance, penalties, and
covering lemma used in the proof.
Section~\ref{sec:preli:copies} sets up the inverse bounds and conditional copies.
Section~\ref{sec:preli:outer} describes the outer copy constructions.

\subsection{Notation}
\label{sec:preli:notation}

All logarithms have base $2$.
For a positive integer $n$, define $G:=\F_2^n$.
The notation $H\leq G$ means that $H$ is a subspace of $G$.
For $A\subseteq G$, define $A+A:=\{a+a':a,a'\in A\}$, and since $G$ has
characteristic $2$ subtraction and addition agree.
For $x\in G$ the set $x+H$ is a coset of $H$, also called a translate of $H$.
For a $G$-valued random variable $X$, let $\HH[X]$ denote its Shannon entropy,
and let $\cU_S$ denote the uniform distribution on a finite nonempty set
$S\subseteq G$.
For jointly distributed random variables, define
\[
\HH[X\mid Z]:=\E_{z\sim Z}[\HH[X\mid Z=z]]
\]
and
\[
\II[X:Y\mid Z]
:=\HH[X\mid Z]+\HH[Y\mid Z]-\HH[X,Y\mid Z]
\]
for conditional entropy and conditional mutual information.
For probability distributions $P$ and $Q$ on $G$, write
$\{x_1,\ldots,x_r\}:=\{x\in G:P(x)>0\}$ and define
\[
\KL(P\,\|\,Q)
:=\sum_{i=1}^rP(x_i)\log\frac{P(x_i)}{Q(x_i)},
\]
with value $+\infty$ if $P(x)>0$ and $Q(x)=0$ for some $x$.
For a function $\tau$ on distributions, define
\[
\tau(X\mid Z):=\E_{z\sim Z}[\tau(X\mid Z=z)].
\]

For distributions $\mu,\nu$ on $G$, the convolution $\mu*\nu$ is the law
of the sum of independent variables with those laws. The notation $\mu^{*k}$
denotes convolution of $k$ independent copies of $\mu$.

\subsection{The covering problem}
\label{sec:preli:covering}

We quantify Marton's covering conjecture through the exponent $C$ governing
the polynomial dependence on the doubling constant $K$.  The optimal value
$C_*$ is the infimum formalized below.

\begin{definition}[Marton's polynomial Freiman--Ruzsa problem~\cite{r99}]
\label{prob:marton}
Determine the infimum of all real numbers $C\geq0$ for which the following
assertion holds.
Suppose $n$ is a positive integer, $A\subseteq \F_2^n$ is nonempty, and
$K\geq1$ satisfies $|A+A|\leq K|A|$.
Then $A$ is covered by at most $2K^C$ translates of a subspace
$H\leq \F_2^n$ with $|H|\leq|A|$.
\end{definition}

\subsection{Entropic Ruzsa distance and penalties}
\label{sec:preli:entropy}

We first recall the entropic analogue of Ruzsa distance.

\begin{definition}[Entropic Ruzsa distance~\cite{t10}]\label{def:ruzsa}
Suppose $X$ and $Y$ are $G$-valued random variables.
Let $X'$ and $Y'$ be independent copies of $X$ and $Y$.
The entropic Ruzsa distance is
\begin{equation*}
\dist[X;Y]
:=
\HH[X'-Y']-\frac12\HH[X]-\frac12\HH[Y].
\end{equation*}
\end{definition}

We also use two conditional forms.

\begin{definition}[Jointly conditioned entropic Ruzsa distance]
\label{def:ruzsa:conditional:joint}
Suppose $X,Y,Z$ are jointly distributed.
The jointly conditioned entropic Ruzsa distance is
\begin{equation*}
\dist[(X;Y)\mid Z]
:=
\E_{z\sim Z}[\dist[X\mid Z=z;Y\mid Z=z]],
\end{equation*}
If the conditioning variable is a tuple, its components are sampled jointly.
\end{definition}

For independent pairs, we average over the two conditioning variables separately.

\begin{definition}[Independently conditioned entropic Ruzsa distance]
\label{def:ruzsa:conditional:independent}
Suppose $(X,Z)$ is independent of $(Y,W)$.
The independently conditioned entropic Ruzsa distance is
\begin{equation*}
\dist[X\mid Z;Y\mid W]
:=
\E_{\substack{z\sim Z\\w\sim W}}[\dist[X\mid Z=z;Y\mid W=w]],
\end{equation*}
where $z$ and $w$ are sampled independently from their marginals.
\end{definition}

Liao's abstract penalty lemma~\cite[Lemma~7]{l24} uses the following notion
of admissibility.

\begin{definition}[Admissible penalty]\label{def:penalty}
A continuous real-valued function $\tau$ on $G$-valued distributions is
admissible if the following properties hold.
Suppose $X$ and $Y$ are independent $G$-valued random variables.
Suppose $Z$ is jointly distributed with $X$.
Then
\begin{align*}
\tau(X+Y)
&\leq \tau(X)+\frac12(\HH[X+Y]-\HH[X]),\\
\tau(X\mid Z)
&\leq \tau(X)+\frac12(\HH[X]-\HH[X\mid Z]),\\
\tau(X+s)&=\tau(X)\qquad\text{for every $s\in G$}.
\end{align*}
\end{definition}

The improved comparison also uses the reverse conditioning inequality.

\begin{definition}[Two-sided-admissible penalty]\label{def:penalty:two:sided}
An admissible penalty $\tau$ is two-sided-admissible if, for every joint
law of $X,D$,
\[
\tau(X)\leq\tau(X\mid D)+\frac12\II[X:D].
\]
Equivalently, its conditioning differences satisfy
$|\tau(X\mid D)-\tau(X)|\leq\frac12\II[X:D]$.
\end{definition}

We state the following dense-slice-to-covering lemma of
Liao~\cite[Lemma~9]{l24}, which isolates the covering argument from Gowers,
Green, Manners, and Tao~\cite[Appendix~B]{ggmt25}.

\begin{lemma}[Dense slice to covering, {\cite[Lemma~9]{l24}}]\label{lem:slice:cover}
Suppose $A\subseteq G$ is nonempty.
Suppose $|A+A|\leq K|A|$.
Suppose $R\geq 1$.
Suppose $V\leq G$ and $t\in G$ satisfy
\[
    |A\cap(V+t)|\geq \frac{\max\{|A|,|V|\}}{R}.
\]
Then $A$ is covered by at most $2KR$ translates of a subspace $V'\leq G$ with $|V'|\leq |A|$.
\end{lemma}

\subsection{Common-penalty inverse bounds and conditional copies}
\label{sec:preli:copies}

We compare general and diagonal inverse bounds for the same penalty.

\begin{definition}[Common-penalty inverse bounds]\label{def:inverse:bounds}
Suppose $\tau$ is admissible on $G$. Define
$M_\tau:=\min_{V\leq G}2\tau(\cU_V)$.
For $c\geq0$, let $\mathsf G_\tau(c)$ denote the assertion
\[
M_\tau\leq\tau(X)+\tau(Y)+c\dist[X;Y]
\quad\text{for all laws $X,Y$ on $G$}.
\]
Let $\mathsf D_\tau(c)$ denote the assertion
\[
M_\tau\leq2\tau(X)+c\dist[X;X]
\quad\text{for all laws $X$ on $G$}.
\]
\end{definition}

Chain powers extend a penalty to arbitrary joint laws on product groups.

\begin{definition}[Chain power of a penalty]\label{def:penalty:chain}
Suppose $\tau$ is a two-sided-admissible penalty on $G$ and $N\geq1$.
For a law $Z=(Z_1,\ldots,Z_N)$ on $G^N$, define
\[
\tau^{[N]}(Z):=\sum_{i=1}^N\tau(Z_i\mid Z_1,\ldots,Z_{i-1}).
\]
The first term is unconditioned. Blocks and their coordinates are ordered
lexicographically when powers are iterated.
\end{definition}

An ensemble records the conditional laws and their probabilities.

\begin{definition}[Posterior ensemble]\label{def:posterior:ensemble}
For a finite joint law $(X,S)$ with $X$ taking values in $G$, its posterior
ensemble is $(p_s,P_s)_s$, where $p_s:=\Pr(S=s)>0$ and
$P_s:=\operatorname{Law}(X\mid S=s)$.
Two ensembles $(p_s,P_s)_s$ and $(q_t,Q_t)_t$ agree up to translation
if their indices can be coupled with these marginals so that $P_s$ is
a translate of $Q_t$ at every pair having positive coupling probability.
Distances between independently drawn ensembles use the product weights
$p_sq_t$.
\end{definition}

Weighted ensembles allow the two posterior families to differ.

\begin{definition}[Weighted ensembles]\label{def:ensemble:weighted}
Suppose $\mathcal F=(p_i,P_i)_{i=1}^r$ is a finite family of laws on $G$
with probabilities $p_i$.
For a second such family $\mathcal G=(q_j,Q_j)_{j=1}^s$, define
\[
T_\tau(\mathcal F):=\sum_{i=1}^r p_i\tau(P_i),\qquad
D(\mathcal F,\mathcal G):=\sum_{i=1}^r\sum_{j=1}^s p_iq_j\dist[P_i;Q_j].
\]
For $0<\omega<1$, their weighted union $\mathcal F_\omega$ assigns
weight $\omega p_i$ to each labeled law $P_i$ and weight
$(1-\omega)q_j$ to each labeled law $Q_j$.
\end{definition}

The following constants specify the hypotheses and their strict upgrades.

\begin{definition}[Comparison constants]\label{def:copy:coefficients}
Define
\begin{align*}
\eta&:=\frac2{11},&\theta&:=\frac{2499999}{2500000},\\
a&:=\frac{5051279389241}{1000000000000},&a'&:=\frac{2525637}{500000},\\
b&:=\frac{2469866735027}{500000000000},&b'&:=\frac{4939729}{1000000}.
\end{align*}
\end{definition}

The simultaneous diagonal and branch comparison uses one common contraction.

\begin{definition}[Diagonal comparison constants]\label{def:diagonal:constants}
Define
\begin{align*}
B&:=\frac{4286816997}{1000000000},&c&:=\frac{857363397}{200000000},\\
B_0&:=\frac{2927}{625},&\vartheta&:=1-10^{-13},\\
C_0&:=1+c=\frac{1057363397}{200000000}.
\end{align*}
\end{definition}

The cyclic comparisons use coefficient pairs on one convex envelope.

\begin{definition}[Simultaneous branch coefficients]\label{def:diagonal:branches}
For $j\in\{0,\ldots,83\}$, let $(\alpha_j,\beta_j)$ be the $j$th
pair selected by the finite rational construction in
Definition~\ref{def:finite:diagonal}.
Lemma~\ref{lem:finite:feasibility} proves that the construction is defined.
The pairs are ordered by increasing first coordinate.
\end{definition}

Binary row spaces describe conditioning tuples. A flag records a
conditioning space and one additional output coordinate.

\begin{definition}[Conditional linear state]\label{def:copy:state}
Suppose $Z_1,\ldots,Z_N$ are independent with law $\mu$ on $G$.
For $v\in\F_2^N$, define $Z_v:=\sum_{i=1}^N v_iZ_i$.
For a row space $W$, let $Z_W$ be the tuple of sums associated with any
basis of $W$. A conditional state is a flag $f=(W\subset U)$, where
$U=\operatorname{span}(W,v)$ and $v\notin W$, representing $Z_v\mid Z_W$.
For a common admissible penalty $\tau$, define
\[
t_\mu:=\HH[\mu],\qquad e_\mu(W):=\HH[Z_W]-\dim(W)t_\mu,
\qquad R_\mu(f):=2\tau(f)-2\tau(\mu)+\HH[f]-t_\mu.
\]
Entropy and penalty of a state are averaged over its conditioning values.
Its support consists of the copy coordinates occurring in $U$.
States on disjoint supports have independent conditioning values.
\end{definition}

The inner comparisons average the two orientations of one cyclic branch.

\begin{definition}[Normalized branch quantities]\label{def:cyclic:copies}
Suppose $A+B+C=0$. Define
\begin{align*}
x&:=\II[A:B],&s&:=\II[A:C]+\II[B:C],\\
t&:=\HH[A\mid C]=\HH[B\mid C],&
m&:=\frac{\HH[A]+\HH[B]}2-t,\qquad e:=\HH[C]-t.
\end{align*}
After fixing $C=c$, sample $Z_1,\ldots,Z_{32}$ independently from
$\mu_c:=\operatorname{Law}(A\mid C=c)$.
For the second orientation sample $\widetilde Z_1,\ldots,\widetilde Z_{32}$
independently from $\operatorname{Law}(B\mid C=c)$.
For a row space $V$ on $(C,Z_1,\ldots,Z_{32})$ and a copy row space $W$, set
\begin{align*}
\mathscr U(V)&:=\tfrac12(\HH[V(C,Z)]+\HH[V(C,\widetilde Z)])-\dim(V)t,\\
\mathscr E(W)&:=\HH[Z_W\mid C]-\dim(W)t,\qquad
\mathscr R_f:=\E_c R_{\mu_c}(f),\\
\mathfrak h_f&:=\mathscr E(U_f)-\mathscr E(W_f).
\end{align*}
For a coefficient row $v$ on $(C,Z_1,\ldots,Z_{32})$, the notation
$Z_v$ also includes its indicated multiple of $C$.
The entropies in $\mathscr U$ are unconditional on $C$.
With additional conditioning, these quantities are first formed for
each conditional cyclic law of positive probability and then averaged.
\end{definition}

The two inner quotient configurations use either conditional copy
entropies or entropies retaining the original conditioning coordinate.

\begin{definition}[Inner cyclic quotients]\label{def:inner:quotient}
Use Definition~\ref{def:cyclic:copies}, and let $f=(W_f\subset U_f)$
be a conditional state. Suppose $K\subset T$ has quotient dimension two,
with distinct intermediate spaces $W_1,W_2,W_3$, and select $i\in\{1,2,3\}$.
Allow either of the following configurations:
\begin{enumerate}
\item All spaces use only copy coordinates,
$T=U_f$, $W_i=W_f$, and $\mathscr F:=\mathscr E$.
\item $T=\operatorname{span}(C,U_f)$,
$W_i=\operatorname{span}(C,W_f)$, $C\notin K$, and
$\mathscr F:=\mathscr U$.
\end{enumerate}
For $\{i,j,k\}=\{1,2,3\}$, define
\begin{align*}
x_i&:=\mathscr F(W_j)+\mathscr F(W_k)-\mathscr F(T)-\mathscr F(K),\\
s_i&:=2\mathscr F(W_i)+\mathscr F(W_j)+\mathscr F(W_k)
       -2\mathscr F(T)-2\mathscr F(K).
\end{align*}
\end{definition}

The branch hypotheses prescribe eleven fixed shapes in the two
nonnegative information quantities $x,s$.

\begin{definition}[Universal branch hypotheses]\label{def:branch:hypotheses}
Let $(u_j,v_j)$, for $0\leq j\leq10$, be the pairs selected in
Definition~\ref{def:finite:general}, and set $q_j:=\theta$.
Lemma~\ref{lem:finite:feasibility} proves that these pairs are defined.
For $r\geq0$, let $\mathsf B_{j,\tau}(r)$ denote the assertion
\begin{equation}\label{eq:branch:hypothesis}
\eta(M_\tau-2\tau(A\mid C))\leq r(u_jx+v_js)
\quad\text{for every law $A+B+C=0$ on $G$},
\end{equation}
where $x,s$ are as in Definition~\ref{def:cyclic:copies}.
Under additional finite conditioning, the assertion is applied to each
conditional cyclic law of positive probability and then averaged,
with the same $M_\tau$.
\end{definition}

The final comparison uses the branch coefficients from
Definition~\ref{def:diagonal:branches} directly.

\begin{definition}[Simultaneous cyclic hypotheses]\label{def:closed:hypotheses}
For $j\in\{0,\ldots,83\}$ and $q\geq0$, let
$\mathsf C_{j,\tau}(q)$ denote the assertion
\[
M_\tau-2\tau(A\mid C)\leq q(\alpha_jx+\beta_js)
\quad\text{for every law $A+B+C=0$ on $G$},
\]
where $x,s$ are as in Definition~\ref{def:cyclic:copies}.
With additional finite conditioning, apply the assertion to each
conditional cyclic law of positive probability and average.
\end{definition}

\subsection{Outer copies and cyclic quotients}
\label{sec:preli:outer}

The two outer comparisons use different sampling symmetries.

\begin{definition}[General outer copies]\label{def:outer:copies}
Suppose $X_1,\ldots,X_8$ have law $X$ and $Y_1,\ldots,Y_8$ have law $Y$,
all mutually independent. Set $t:=(\HH[X]+\HH[Y])/2$,
$d:=\dist[X;Y]$, and $\tauzero:=\tau(X)+\tau(Y)$.
The subscript $\mathrm{sym}$ denotes the average with the value obtained
by globally interchanging the laws $X$ and $Y$.
For a coefficient space $U$, define
$\mathcal E(U):=\HH_{\mathrm{sym}}[Z_U]-\dim(U)t$.
The average is taken after evaluating each entropy or penalty.
\end{definition}

For the diagonal comparison, all outer copies have the same law.

\begin{definition}[Diagonal outer copies]\label{def:outer:diagonal}
Suppose $Z_1,\ldots,Z_{36}$ are independent with law $X$.
Set $t:=\HH[X]$, $d:=\dist[X;X]$, and $\tauzero:=2\tau(X)$.
For a coefficient space $U$, define
$\mathcal E(U):=\HH[Z_U]-\dim(U)t$.
In this setting the subscript $\mathrm{sym}$ leaves a quantity unchanged.
\end{definition}

The same state normalization applies to both sampling conventions.

\begin{definition}[Outer conditional state]\label{def:outer:state}
Under Definition~\ref{def:outer:copies} or~\ref{def:outer:diagonal},
let $f=(W\subset U)$ with $\dim(U/W)=1$ represent $Z_v\mid Z_W$,
where $v\in U\setminus W$. Define
\[
h_f:=\mathcal E(U)-\mathcal E(W),\qquad
\mathcal R(f):=2\tau_{\mathrm{sym}}(f)-\tauzero
 +\HH_{\mathrm{sym}}[f]-t.
\]
The support is the set of copy coordinates occurring in $U$.
Conditioning values of states on disjoint supports are independent.
\end{definition}

A two-dimensional quotient gives three cyclic branches.

\begin{definition}[Cyclic quotient budgets]\label{def:outer:quotient}
Suppose $K\subset U$ are outer coefficient spaces with $\dim(U/K)=2$.
Let $W_1,W_2,W_3$ be their distinct intermediate spaces and set
$f_i:=(W_i\subset U)$. For $\{i,j,k\}=\{1,2,3\}$ define
\begin{align*}
x_i&:=\mathcal E(W_j)+\mathcal E(W_k)-\mathcal E(U)-\mathcal E(K),\\
s_i&:=2\mathcal E(W_i)+\mathcal E(W_j)+\mathcal E(W_k)
       -2\mathcal E(U)-2\mathcal E(K),\\
\pi_i&:=\mathcal R(f_i)-\mathcal E(U)+\mathcal E(W_i).
\end{align*}
\end{definition}

%% file: 31_upper.tex
\section{Proof of the Upper Bound}\label{sec:proofs}

Section~\ref{sec:proofs:tensor} establishes the penalty and ensemble tools.
Section~\ref{sec:proofs:convolution} gives the entropy and penalty identities
used in the comparisons.
Section~\ref{sec:proofs:certificate} develops the quotient comparisons.
Section~\ref{sec:proofs:diagonal:refinement} proves the diagonal inverse bound.
Section~\ref{sec:proofs:consequences} deduces the entropic and covering
consequences, completing the proof of Theorem~\ref{thm:main:formal}.

\subsection{Chain powers and posterior ensembles}
\label{sec:proofs:tensor}

The two-sided conditioning inequality also controls continuity and
convolution. We first record these consequences.

\begin{lemma}[Consequences of two-sided conditioning]
\label{lem:penalty:core}
Suppose $\tau$ is a real-valued function on distributions on a finite
alphabet $\Omega=\{\omega_1,\ldots,\omega_q\}$. Define
$\Phi_+(P):=\HH[P]+2\tau(P)$ and $\Phi_-(P):=\HH[P]-2\tau(P)$.
The inequality
\begin{equation}\label{eq:core:two:sided}
|\tau(X\mid D)-\tau(X)|\leq\tfrac12\II[X:D]
\end{equation}
holds for every joint law of $X,D$ with $X$ taking values in $\Omega$
if and only if both $\Phi_+$ and $\Phi_-$ are concave.
In this case, $\tau$ is continuous.
If $\Omega=G$ and $\tau$ is translation invariant, then $\tau$ is
two-sided-admissible.

Suppose further that $\tau$ is two-sided-admissible on $G$.
Further conditioning of $g$ to obtain $f$, up to a known translate, gives
\begin{equation}\label{eq:core:conditioning}
0\leq\Phi_+(g)-\Phi_+(f)\leq2(\HH[g]-\HH[f]).
\end{equation}
A conditional convolution $f=g*h$ on disjoint supports, with
independently drawn conditioning values, gives
\begin{equation}\label{eq:core:convolution}
0\leq\Phi_+(f)-\Phi_+(g)\leq2(\HH[f]-\HH[g]).
\end{equation}
The entropy and penalty of each state are averaged over its conditioning
values.
\end{lemma}
\begin{proof}
For the mixture of the laws of $X$ given $D$, concavity of $\Phi_+$ and
$\Phi_-$ is equivalent to
\[
\II[X:D]\pm2(\tau(X)-\tau(X\mid D))\geq0.
\]
These are exactly the two inequalities in Eq.~\eqref{eq:core:two:sided}.

Write $\delta_x$ for the point mass at $x$.
Conditioning a variable $X$ with law $P$ on itself gives
\[
|\tau(P)-\sum_{i=1}^qP(\omega_i)\tau(\delta_{\omega_i})|\leq\tfrac12\HH[P].
\]
Since $\HH[P]\leq\log|\Omega|$, the range of $\tau$ has width at most
$B:=\max_x\tau(\delta_x)-\min_x\tau(\delta_x)+\log|\Omega|$.
For laws $P,Q$ at total variation distance $\delta\in(0,1)$, their
common part gives decompositions
$P=(1-\delta)R+\delta P_1$ and $Q=(1-\delta)R+\delta Q_1$.
Let $h_2$ denote binary entropy.
\begin{samepage}
We have
\[
|\tau(P)-\tau(Q)|
\leq\delta|\tau(P_1)-\tau(Q_1)|+h_2(\delta)
\leq\delta B+h_2(\delta).
\]
The first step follows from Eq.~\eqref{eq:core:two:sided} applied to
each binary mixture label and the triangle inequality.
The last step follows from the bound on the range of $\tau$.
\par
\end{samepage}
The resulting bound tends to zero with $\delta$, proving continuity.

For independent $X,Y$, translation invariance gives
$\tau(X+Y\mid Y)=\tau(X)$.
\begin{samepage}
We have
\[
|2\tau(X+Y)-2\tau(X)|
\leq\II[X+Y:Y]=\HH[X+Y]-\HH[X].
\]
The first step follows from Eq.~\eqref{eq:core:two:sided}.
The last step follows from independence and translation invariance
of entropy.
\par
\end{samepage}
This proves convolution admissibility. Together with continuity,
translation invariance, and Eq.~\eqref{eq:core:two:sided}, it proves
two-sided admissibility.
Adding the entropy difference to Eq.~\eqref{eq:core:two:sided}
gives Eq.~\eqref{eq:core:conditioning}; adding it to the preceding
convolution bound gives Eq.~\eqref{eq:core:convolution}.
For conditional states, first fix the earlier conditioning values
and then average. Known conditional translates preserve both quantities.
\end{proof}

The conditioning chain rule preserves the penalty class on product groups.
Uniform subgroup fibres determine its minimum exactly.

\begin{lemma}[Chain powers and subgroup minima]\label{lem:penalty:chain}
Suppose $\tau$ is two-sided-admissible on $G$. Every $\tau^{[N]}$
is two-sided-admissible on $G^N$. For positive integers $N,L$ and
$\lambda\geq1$,
\begin{equation}\label{eq:chain:identities}
M_{\tau^{[N]}}=NM_\tau,\qquad (\tau^{[N]})^{[L]}=\tau^{[NL]},\qquad (\tau/\lambda)^{[N]}=\tau^{[N]}/\lambda.
\end{equation}
For independent coordinates with laws $P_1,\ldots,P_N$,
\begin{equation}\label{eq:chain:product}
\tau^{[N]}(P_1\otimes\cdots\otimes P_N)=\sum_{i=1}^N\tau(P_i).
\end{equation}
\end{lemma}
\begin{proof}
A fixed translation reindexes the prefixes and translates every
conditional coordinate, so $\tau^{[N]}$ is translation invariant.
Write $Z_{<i}:=(Z_1,\ldots,Z_{i-1})$.
\begin{samepage}
We have
\begin{align}
|\tau^{[N]}(Z\mid D)-\tau^{[N]}(Z)|
&\leq\sum_{i=1}^N
 |\tau(Z_i\mid Z_{<i},D)-\tau(Z_i\mid Z_{<i})|\notag\\
&\leq\tfrac12\sum_{i=1}^N\II[Z_i:D\mid Z_{<i}]\notag\\
&=\tfrac12\II[Z:D].\label{eq:chain:conditioning}
\end{align}
The first step follows from the triangle inequality.
The second step follows from two-sided conditioning at each fixed
prefix and averaging.
The last step follows from the mutual-information chain rule.
\par
\end{samepage}
Lemma~\ref{lem:penalty:core} gives continuity and convolution
admissibility. Thus $\tau^{[N]}$ is two-sided-admissible.
Expanding the coordinate chains proves the second identity in
Eq.~\eqref{eq:chain:identities}. Linearity proves the third identity,
and independence proves Eq.~\eqref{eq:chain:product}.

Suppose $Z$ is uniform on $H\leq G^N$.
Its $i$th coordinate conditional on any prefix of positive probability
is uniform on a coset of the projection $V_i$ of
$\{z\in H:z_1=\cdots=z_{i-1}=0\}$ onto coordinate $i$.
All fibres of this projection have equal size. Translation invariance gives
\begin{equation}\label{eq:chain:fibre:lower}
2\tau^{[N]}(\cU_H)=\sum_{i=1}^N2\tau(\cU_{V_i})\geq NM_\tau.
\end{equation}
Conversely, if $V_*$ attains $M_\tau$, the product subgroup $V_*^N$
attains equality by Eq.~\eqref{eq:chain:product}.
This proves the first identity in Eq.~\eqref{eq:chain:identities}.
\end{proof}

A diagonal bound on all chain powers controls independent draws from one
finite ensemble, even when their realized laws differ.

\begin{lemma}[Averaged diagonal inverse bound]\label{lem:inverse:ensemble}
Suppose $\tau$ is two-sided-admissible and $b>0$.
Assume $\mathsf D_{\tau^{[N]}}(b)$ for every $N\geq1$.
Then every finite ensemble $(p_i,P_i)_{i=1}^m$ of laws on $G$ satisfies
\begin{equation}\label{eq:ensemble:inverse}
M_\tau\leq2\sum_{i=1}^m p_i\tau(P_i)
 +b\sum_{i=1}^m\sum_{j=1}^m p_ip_j\dist[P_i;P_j].
\end{equation}
If two posterior ensembles $f,g$ agree up to translation in
Definition~\ref{def:posterior:ensemble}, then
\begin{equation}\label{eq:ensemble:pair}
M_\tau\leq\tau(f)+\tau(g)+b\dist[f;g],
\end{equation}
where the conditioning values in the distance are independent.
\end{lemma}
\begin{proof}
Choose $I\in\{1,\ldots,m\}$ with probabilities $p_i$ and, conditional on $I=i$, give
$Z_N$ the product law $P_i^{\otimes N}$. One index is shared by all
coordinates. Set
\[
T:=\sum_{i=1}^m p_i\tau(P_i),\qquad
D:=\sum_{i=1}^m\sum_{j=1}^m p_ip_j\dist[P_i;P_j],\qquad h_I:=\HH[I].
\]
Lemma~\ref{lem:penalty:chain} gives $\tau^{[N]}(Z_N\mid I)=NT$.
Reverse conditioning and $\II[Z_N:I]\leq h_I$ therefore give
\begin{equation}\label{eq:ensemble:penalty}
\tau^{[N]}(Z_N)\leq NT+\tfrac12h_I.
\end{equation}

Take an independent copy $(J,Z_N')$ of $(I,Z_N)$.
Conditional on $(I,J)=(i,j)$, their sum has law
$(P_i*P_j)^{\otimes N}$.
Writing $h:=\sum_{i=1}^m p_i\HH[P_i]$, the entropy chain rule gives
\[
\HH[Z_N+Z_N']\leq N\sum_{i=1}^m\sum_{j=1}^m p_ip_j\HH[P_i*P_j]+2h_I,
\qquad \HH[Z_N]\geq Nh.
\]
Subtract the second bound from the first. Since the marginal entropy
terms in $D$ sum to $h$, this gives
\begin{equation}\label{eq:ensemble:distance}
\dist[Z_N;Z_N]\leq ND+2h_I.
\end{equation}
Lemma~\ref{lem:penalty:chain} also gives $M_{\tau^{[N]}}=NM_\tau$.
\begin{samepage}
We have
\[
NM_\tau
\leq2\tau^{[N]}(Z_N)+b\dist[Z_N;Z_N]
\leq2NT+h_I+b\dist[Z_N;Z_N]
\leq2NT+h_I+b(ND+2h_I).
\]
The first step is the diagonal hypothesis for $\tau^{[N]}$.
The second step uses Eq.~\eqref{eq:ensemble:penalty}.
The last step uses Eq.~\eqref{eq:ensemble:distance}.
\par
\end{samepage}
After division by $N$, the right side is
$2T+bD+(1+2b)h_I/N$. The ensemble is fixed and finite, so letting
$N$ tend to infinity proves Eq.~\eqref{eq:ensemble:inverse}.

For Eq.~\eqref{eq:ensemble:pair}, use the coupling in
Definition~\ref{def:posterior:ensemble}. Translation invariance gives
$\tau(f)=\tau(g)$. Two independent copies of that coupling identify
the averaged distances from $f$ to $f$ and from $f$ to $g$.
Apply Eq.~\eqref{eq:ensemble:inverse} to the ensemble $f$.
\end{proof}

\subsection{Entropy and penalty identities}
\label{sec:proofs:convolution}

Entropy submodularity applies to overlapping tuples through their common
linear functions.

\begin{lemma}[Common-function inequality]\label{lem:entropy:common:function}
Suppose $F$ is a deterministic function of each of the finite random
variables $P,Q$. Then
\begin{equation}\label{eq:entropy:common:function}
\HH[P]+\HH[Q]\geq\HH[F]+\HH[P,Q].
\end{equation}
\end{lemma}
\begin{proof}
Subtract $2\HH[F]$ and apply conditional subadditivity given $F$.
\end{proof}

The entropy term in $R_\mu$ cancels the admissibility charge for conditioning.

\begin{lemma}[Conditional-state rules]\label{lem:copy:rules}
Suppose $\tau$ is two-sided-admissible.
Under Definition~\ref{def:copy:state}, the following assertions hold.
\begin{enumerate}
\item For $f=\mu^{*k}$, $k\geq1$,
$R_\mu(f)\leq2(\HH[\mu^{*k}]-t_\mu)$.
\item If further conditioning of $g$ produces $f$, up to a translate
determined by the conditioning values, then $R_\mu(f)\leq R_\mu(g)$.
\item If $f=g*h$ on disjoint supports, with independently drawn
conditioning values, then
$R_\mu(f)\leq R_\mu(g)+2(\HH[f]-\HH[g])$.
Either factor may be chosen as $g$.
\item An independent conditioning component may be removed if its
output contribution is determined by that component. For the resulting
state $g$, $R_\mu(f)=R_\mu(g)$.
\item For $f=(W\subset U)$ and $v\in U\setminus W$,
$R_\mu(f)\leq2e_\mu(\operatorname{span}(v))$.
\item If further conditioning of $g$ produces $f$, up to a known
conditional translate, then
\begin{equation}\label{eq:copy:reverse:conditioning}
R_\mu(f)-R_\mu(g)+2(\HH[g]-\HH[f])\geq0.
\end{equation}
\end{enumerate}
In the second and sixth assertions, every subspace $W_0\subseteq W$
is allowed as earlier conditioning: for $v\in U\setminus W$, take
$f:=Z_v\mid Z_W$ and $g:=Z_v\mid Z_{W_0}$.
Moreover,
\begin{equation}\label{eq:copy:state:entropy}
\HH[f]-t_\mu=e_\mu(U_f)-e_\mu(W_f).
\end{equation}
\end{lemma}
\begin{proof}
The entropy chain rule gives Eq.~\eqref{eq:copy:state:entropy}.
With the notation of Lemma~\ref{lem:penalty:core},
Definition~\ref{def:copy:state} gives
\[
R_\mu(f)=\Phi_+(f)-\Phi_+(\mu).
\]
For the first assertion, apply Eq.~\eqref{eq:core:convolution} to
one copy of $\mu$ and the sum of the remaining copies; for $k=1$
both sides vanish.
For the second and sixth assertions, apply
Eq.~\eqref{eq:core:conditioning}.
For the third assertion, apply Eq.~\eqref{eq:core:convolution}.
The fourth assertion follows from independence and translation invariance.
For the fifth assertion, apply the first assertion to $Z_v$ and then
the second assertion to condition on $Z_W$.
Each comparison is first made at fixed conditioning values and then
averaged. Since $Z_W$ determines $Z_{W_0}$ for every $W_0\subseteq W$,
the second and sixth assertions allow every stated earlier conditioning.
\end{proof}

Copy permutations identify the posterior ensembles used by the inverse
costs while preserving their probability weights.

\begin{lemma}[Matching posterior ensembles]\label{lem:ensemble:flags}
Suppose $f,g$ are conditional linear states on disjoint supports of
independent input copies. Remove independent irrelevant conditioning
components as in Lemma~\ref{lem:copy:rules}.
If a permutation of the remaining copies maps the conditioning and output
spaces of $f$ to those of $g$ and preserves each input law, their posterior
ensembles agree up to translation.
For the general outer copies, the permutation must preserve the $X$
and $Y$ coordinates separately.
\end{lemma}
\begin{proof}
The permutation preserves the joint law of the independent inputs.
It takes a basis of the first conditioning space to a basis of the
second, up to an invertible change of coordinates. Thus it gives a
probability-preserving correspondence of conditioning values.
The mapped output differs from the second output by a row of its
conditioning space, since that space has codimension one in the
output space. This difference is a known translate at each conditioning
value. The correspondence therefore supplies the coupling in
Definition~\ref{def:posterior:ensemble}.
The removed components are independent of the retained copies and
contribute only known output translates, so restoring them preserves
the agreement of ensembles. Disjoint supports make the two actual
conditioning draws independent.
In the general outer setting, preservation of each input law requires
the stated separate preservation of the two coordinate classes.
\end{proof}

Two distinct scalar coordinates describe the entropy budget of one branch.

\begin{lemma}[Branch normalization]\label{lem:copy:normalization}
Under Definition~\ref{def:cyclic:copies},
\begin{align}
m&=s/2,\qquad e=s-x,\label{eq:copy:me}\\
\mathscr U(C)&=e,\qquad
\mathscr U(Z_i)=\mathscr U(C+Z_i)=m,\label{eq:copy:singletons}\\
\mathscr U(C,Z_W)&=e+\mathscr E(W).\label{eq:copy:with:C}
\end{align}
The functional $\mathscr E$ is additive on disjoint copy supports and
vanishes on one copy. Both entropy functionals are invariant under row
operations and copy permutations. The functional $\mathscr U$ is also
invariant under simultaneous substitution $Z_i\mapsto Z_i+C$.
For row spaces $P,Q$ on $(C,Z)$,
\begin{equation}\label{eq:copy:shannon}
\mathscr U(P)+\mathscr U(Q)-\mathscr U(P\cap Q)-\mathscr U(P+Q)\geq0.
\end{equation}
For row spaces $P,Q$ on the copies alone,
\begin{equation}\label{eq:copy:conditional:shannon}
\mathscr E(P)+\mathscr E(Q)-\mathscr E(P\cap Q)-\mathscr E(P+Q)\geq0.
\end{equation}
If a row space $P$ on $(C,Z)$ does not contain the coefficient row of
$C$, let $P^\circ$ be its projection onto the copy coordinates. Then
\begin{equation}\label{eq:copy:entropy:conditioning}
\mathscr U(P)-\mathscr E(P^\circ)\geq0.
\end{equation}
The two orientations have the same averaged conditional-state entropies,
penalties, and inverse costs.
\end{lemma}
\begin{proof}
Every pair of $A,B,C$ determines the third. If $h:=\HH[A,B]$, then
$t=h-\HH[C]$,
$s=\HH[A]+\HH[B]+2\HH[C]-2h$, and
$x=\HH[A]+\HH[B]-h$. These identities prove Eq.~\eqref{eq:copy:me}.
The unconditional copy marginals are $A$ and $B$ in the two orientations.
Adding $C$ interchanges them. This proves Eq.~\eqref{eq:copy:singletons}.
The chain rule proves Eq.~\eqref{eq:copy:with:C}.

Conditional independence gives additivity of $\mathscr E$, and one copy
has conditional entropy $t$. Row operations preserve each tuple, and
copy permutations preserve its conditional law. Since $B=A+C$,
simultaneous translation of the copies by $C$ interchanges the two
orientations, proving the additional symmetry of $\mathscr U$.
At each fixed $C=c$, this translation changes a state only by deterministic
translations of its conditioning tuple and output. It therefore preserves
the averaged conditional quantities and inverse costs.
Finally apply Lemma~\ref{lem:entropy:common:function} to the tuples with
spaces $P,Q$ in each orientation. The dimension identity for intersection
and sum cancels every rank term and gives Eq.~\eqref{eq:copy:shannon}.
Applying the same argument after fixing $C=c$ and then averaging proves
Eq.~\eqref{eq:copy:conditional:shannon}.
For Eq.~\eqref{eq:copy:entropy:conditioning}, the projection from $P$
to $P^\circ$ is injective, so their dimensions agree.
Conditioning on $C$ decreases each orientation's entropy.
At a fixed value of $C$, the tuple represented by $P$ is a translate
of the tuple represented by $P^\circ$.
The two orientations have the same averaged conditional entropy,
and the equal rank terms cancel.
\end{proof}

Two conditionally independent copies give a finite initial bound for
every branch inequality. This supplies a uniform initial scale for
the simultaneous improvements.

\begin{lemma}[Finite initial branch bound]\label{lem:branch:initial}
Suppose $\tau$ is admissible and $\mathsf G_\tau(\kappa_0)$ holds for
some finite $\kappa_0\geq0$.
Under Definition~\ref{def:cyclic:copies}, set
$p_2:=\HH[Z_1+Z_2\mid C]-t$. Then
\begin{align}
0\leq p_2&\leq x+s,\label{eq:branch:initial:entropy}\\
\eta(M_\tau-2\tau(A\mid C))
&\leq\kappa_0\eta(x+s).\label{eq:branch:initial:inverse}
\end{align}
\end{lemma}
\begin{proof}
Set $c:=\HH[C]$, $P:=(Z_1,Z_2)$, and $Q:=(C+Z_1,C+Z_2)$.
The sum $Z_1+Z_2$ is a deterministic function of each of $P,Q$.
Their joint tuple determines and is determined by $(C,Z_1,Z_2)$,
which has entropy $c+2t$ by conditional independence.
Each coordinate of $P$ has marginal law $A$, and each coordinate
of $Q$ has marginal law $B$.
\begin{samepage}
We have
\[
p_2+c+3t\leq\HH[Z_1+Z_2]+\HH[P,Q]\leq\HH[P]+\HH[Q]\leq2\HH[A]+2\HH[B].
\]
The first step follows from the definition of $p_2$, conditioning
decreasing entropy, and $\HH[P,Q]=c+2t$.
The second step follows from Lemma~\ref{lem:entropy:common:function}.
The last step follows from subadditivity.
\par
\end{samepage}
The definition of $x$ gives $\HH[A]+\HH[B]=c+t+x$.
Lemma~\ref{lem:copy:normalization} gives $c-t=s-x$.
Substituting these identities into the preceding bound gives
$p_2\leq c-t+2x=x+s$.
Nonnegativity follows because independent convolution does not decrease
entropy at each fixed $C$. This proves Eq.~\eqref{eq:branch:initial:entropy}.

\begin{samepage}
We have
\[
\eta(M_\tau-2\tau(A\mid C))
\leq\kappa_0\eta p_2\leq\kappa_0\eta(x+s).
\]
The first step follows by applying $\mathsf G_\tau(\kappa_0)$ to two
copies of $A\mid C=c$ at every conditioning value of positive
probability and averaging. Their averaged self-distance is $p_2$.
The last step follows from Eq.~\eqref{eq:branch:initial:entropy}.
\par
\end{samepage}
This proves Eq.~\eqref{eq:branch:initial:inverse}.
\end{proof}

\subsection{Cyclic quotients and outer copies}
\label{sec:proofs:certificate}

A cyclic inequality applies to every conditional quotient with the
selected posterior. The same substitution covers both inner configurations
and the outer copies.

\begin{lemma}[Cyclic quotient substitution]\label{lem:cyclic:substitution}
Suppose $\tau$ is two-sided-admissible and $a,b\geq0$ satisfy
\[
M_\tau-2\tau(A\mid C)\leq ax+bs
\quad\text{for every law $A+B+C=0$ on $G$}.
\]
For either configuration in Definition~\ref{def:inner:quotient},
the quantities $x_i,s_i$ are nonnegative and
\begin{equation}\label{eq:cyclic:substitution:inner}
M_\tau-2\tau(A\mid C)
\leq ax_i+bs_i+\mathscr R_f-\mathfrak h_f.
\end{equation}
For a quotient in Definition~\ref{def:outer:quotient}, they are
nonnegative and
\begin{equation}\label{eq:cyclic:substitution:outer}
M_\tau-\tauzero\leq ax_i+bs_i+\pi_i.
\end{equation}
Both statements hold under additional finite conditioning.
\end{lemma}
\begin{proof}
Choose generators $a_0,b_0$ of the two-dimensional quotient so that
$a_0+b_0$ generates $W_i/K$. Set
$A_0:=Z_{a_0}$, $B_0:=Z_{b_0}$, and $C_0:=Z_{a_0+b_0}$.
Let $D:=(C,Z_K)$ in the first inner configuration, and $D:=Z_K$
in the second inner configuration and in the outer setting.
At each value of $D$, the three outputs form a cyclic law.
The selected posterior $A_0\mid C_0,D$ is $f$ (or $f_i$ outside),
up to a translate determined by the conditioning values: any two
output rows outside $W_i$ differ by a row of $W_i$.
In the second inner configuration, $W_i$ contains the original $C$,
so this posterior also conditions on $C$.

The entropy chain rule identifies $x_i$ with
$\II[A_0:B_0\mid D]$ and $s_i$ with
$\II[A_0:C_0\mid D]+\II[B_0:C_0\mid D]$, averaged over the
orientations specified by the corresponding definition.
For example, the first mutual information expands as the entropies of
$W_j,W_k$ minus those of the top space and $K$.
All rank terms cancel. The other two expansions give $s_i$, proving
both the stated formulas and nonnegativity.

Apply the assumed cyclic inequality at each value of $D$ and average.
Translation invariance identifies its penalty term with the selected
posterior, giving
\[
M_\tau\leq ax_i+bs_i+2\tau(f),
\]
with the appropriate orientation average understood.
In the inner setting,
$\mathscr R_f-\mathfrak h_f=2\tau(f)-2\tau(A\mid C)$.
In the outer setting, $\pi_i=2\tau_{\mathrm{sym}}(f_i)-\tauzero$.
Subtracting the respective reference penalty gives
Eqs.~\eqref{eq:cyclic:substitution:inner} and~\eqref{eq:cyclic:substitution:outer}.
Further finite conditioning is handled by the same pointwise application
and averaging.
\end{proof}

The general and diagonal normalizations obey the same entropy and penalty
rules, with their respective copy symmetries.

\begin{lemma}[Outer normalization rules]\label{lem:outer:rules}
Suppose $\tau$ is two-sided-admissible.
Use Definition~\ref{def:outer:copies} or~\ref{def:outer:diagonal} and the
states in Definition~\ref{def:outer:state}. For coefficient spaces $P,Q$,
\begin{equation}\label{eq:outer:shannon}
\mathcal E(P)+\mathcal E(Q)-\mathcal E(P\cap Q)-\mathcal E(P+Q)\geq0.
\end{equation}
The functional $\mathcal E$ is additive on disjoint supports and vanishes
on one input coordinate. It is invariant under row operations and
permutations within each law, and under global interchange of the two
laws in the general setting. In the diagonal setting arbitrary copy
permutations are permitted. Moreover,
$h_f=\HH_{\mathrm{sym}}[f]-t$, and the following rules hold.
\begin{enumerate}
\item An unconditioned sum with output space $U$ satisfies
$\mathcal R(f)\leq2\mathcal E(U)$. For one copy, $\mathcal R(f)=0$.
\item Further conditioning of $g$ to obtain $f$, up to a known translate,
gives $\mathcal R(f)\leq\mathcal R(g)$.
\item A conditional convolution $f=g*h$ on disjoint supports gives
$\mathcal R(f)\leq\mathcal R(g)+2(h_f-h_g)$.
\item Removing independent irrelevant conditioning as in
Lemma~\ref{lem:copy:rules} preserves $\mathcal R$.
\item For $f=(W\subset U)$ and $v\in U\setminus W$,
$\mathcal R(f)\leq2\mathcal E(\operatorname{span}(v))$.
\item Further conditioning of $g$ to obtain $f$, up to a known translate,
gives
\begin{equation}\label{eq:outer:reverse:conditioning}
\mathcal R(f)-\mathcal R(g)+2(h_g-h_f)\geq0.
\end{equation}
\item A conditional convolution $f=g*h$ on disjoint supports gives
$\mathcal R(f)-\mathcal R(g)\geq0$.
\end{enumerate}
\end{lemma}
\begin{proof}
Apply Lemma~\ref{lem:entropy:common:function} to the two tuples and their
intersection, and average orientations when appropriate. The dimension
identity cancels all rank terms, giving Eq.~\eqref{eq:outer:shannon}.
Independence gives additivity on disjoint supports. One coordinate has
averaged entropy $t$. Row operations and the permitted permutations
preserve the tuple law or interchange the two orientations.
The chain rule gives $h_f=\HH_{\mathrm{sym}}[f]-t$.

Let $\Phi_{+,\mathrm{sym}}$ denote the orientation average of
$\Phi_+$ from Lemma~\ref{lem:penalty:core}.
Definition~\ref{def:outer:state} gives
\[
\mathcal R(f)=\Phi_{+,\mathrm{sym}}(f)-(\tauzero+t).
\]
For the first rule, apply Eq.~\eqref{eq:core:convolution} to one
input copy and the remaining sum, then average orientations.
That copy has averaged $\Phi_+$ equal to $\tauzero+t$.
For the second and sixth rules, apply Eq.~\eqref{eq:core:conditioning}
at each earlier conditioning value and average orientations.
For the third and seventh rules, apply Eq.~\eqref{eq:core:convolution}
at each pair of conditioning values and average orientations.
The fourth rule follows from independence and translation invariance.
For the fifth rule, apply the first rule to $Z_v$ and then the second
rule to condition on $Z_W$.
\end{proof}

\subsection{Simultaneous diagonal and branch improvement}
\label{sec:proofs:diagonal:refinement}

Convex combinations of cyclic bounds supply all coefficient pairs above
their lower convex envelope.

\begin{lemma}[Convex reduction of cyclic bounds]\label{lem:cyclic:convex}
Suppose $\mathsf C_{j,\rho}(1)$ holds for $0\leq j\leq83$.
If $a,b\geq0$ and there are weights $\lambda_j\geq0$ such that
\[
\sum_{j=0}^{83}\lambda_j=1,\qquad
\sum_{j=0}^{83}\lambda_j\alpha_j\leq a,\qquad
\sum_{j=0}^{83}\lambda_j\beta_j\leq b,
\]
then every cyclic law satisfies $M_\rho-2\rho(A\mid C)\leq ax+bs$.
The same conclusion holds under further finite conditioning.
\end{lemma}
\begin{proof}
Multiply the $j$th hypothesis by $\lambda_j$ and sum over $j$.
Since $x,s\geq0$, the coefficient inequalities give the assertion.
Apply this argument to each conditional law and average.
\end{proof}

The next comparison strictly improves the diagonal and cyclic bounds
together. Each source bound follows from the hypotheses by convex reduction.

\begin{lemma}[Simultaneous diagonal and cyclic comparison]
\label{lem:diagonal:comparison}
Suppose $\tau$ is two-sided-admissible and, for every $N\geq1$,
$\mathsf D_{\tau^{[N]}}(B)$ and $\mathsf C_{j,\tau^{[N]}}(1)$ hold
for every $j\in\{0,\ldots,83\}$.
Then $\mathsf D_{\tau^{[N]}}(c)$ and
$\mathsf C_{j,\tau^{[N]}}(\vartheta)$ hold for every such $N,j$, where
\begin{equation}\label{eq:closed:contraction}
0<c/B\leq\vartheta=1-10^{-13}<1.
\end{equation}
\end{lemma}
\begin{proof}
First work with $\tau$ on $G$. Define
$L_{\mathrm{in}}:=M_\tau-2\tau(A\mid C)$ and
$L_{\mathrm{out}}:=M_\tau-2\tau(X)$.
Lemma~\ref{lem:finite:feasibility} supplies the coefficient tuple and
the 85 rational identities in Definition~\ref{def:finite:diagonal}.
Every licensed source pair is above a convex combination of the
hypothesized pairs, so Lemma~\ref{lem:cyclic:convex} supplies its cyclic
bound. Every inner inverse coefficient lies in $[B,B_0]$.
Nonnegativity of Ruzsa distance therefore lets
$\mathsf D_\tau(B)$ supply every licensed diagonal hypothesis.

The inner and outer row families use 32 fresh conditional copies and
36 original copies, respectively. Their formal normalizations and all
entropy, penalty, inverse, and quotient rows are specified in
Definitions~\ref{def:finite:coordinates} and~\ref{def:finite:rows}.
Lemma~\ref{lem:finite:validity} makes every row nonnegative under the
present hypotheses.
Let $m_j$ and $m_{\mathrm{out}}$ be the numbers of rows in the selected
$j$th inner identity and outer identity, respectively.
The selected identities have the form
\[
\vartheta\alpha_jx+\vartheta\beta_js-L_{\mathrm{in}}=\sum_{r=1}^{m_j} w_{jr}\mathcal I_{jr},\qquad
cd-L_{\mathrm{out}}=\sum_{r=1}^{m_{\mathrm{out}}} w_r\mathcal J_r,
\]
where all weights are nonnegative and the weights on the inverse and
quotient costs sum to one. All auxiliary entropy and penalty
coefficients cancel by the exact cone test in
Definition~\ref{def:finite:cone}.
The first identity proves $\mathsf C_{j,\tau}(\vartheta)$ for every $j$.
The second proves $\mathsf D_\tau(c)$, including when $d=0$.

For general $N$, apply the same argument to $\rho:=\tau^{[N]}$.
Lemma~\ref{lem:penalty:chain} gives $\rho^{[m]}=\tau^{[Nm]}$.
Thus the diagonal ensemble premises and cyclic source hypotheses hold
on all further chain powers of $\rho$. The conclusions hold
simultaneously for every $N$, and under additional finite conditioning
by applying the argument to each conditional law and averaging.
The constants in Definition~\ref{def:diagonal:constants} give
$c/B\leq\vartheta<1$.
\end{proof}

A strict improvement of homogeneous bounds can be iterated
from any finite initial scale. The index below may include every chain power
and every input law.

\begin{lemma}[Finite scaling]\label{lem:finite:scaling}
Suppose $\mathcal T$ is a class of penalties closed under
$\rho\mapsto\rho/\lambda$ for $\lambda\geq1$.
For each $\rho\in\mathcal T$, let $\mathcal L_\rho(\xi)$ be real
quantities indexed by a common set, with nonnegative budgets $Q_\xi$,
such that $\mathcal L_{\rho/\lambda}(\xi)=\mathcal L_\rho(\xi)/\lambda$.
Fix $0<q<1$. Assume that, for every $\rho\in\mathcal T$,
\[
\{\mathcal L_\rho(\xi)\leq Q_\xi\text{ for all }\xi\}
\ \Longrightarrow\
\{\mathcal L_\rho(\xi)\leq qQ_\xi\text{ for all }\xi\}.
\]
If $\tau\in\mathcal T$ and
$\mathcal L_\tau(\xi)\leq\Lambda Q_\xi$ for all $\xi$ and some
finite $\Lambda\geq1$, then
$\mathcal L_\tau(\xi)\leq Q_\xi$ for all $\xi$.
\end{lemma}
\begin{proof}
Suppose the bounds hold at a scale $\lambda\geq1$.
Homogeneity gives the unit-scale hypotheses for $\tau/\lambda$.
Apply the strict improvement and multiply by $\lambda$ to obtain
$\mathcal L_\tau(\xi)\leq q\lambda Q_\xi$ for every $\xi$.
Since $Q_\xi\geq0$, this implies the bounds at scale
$\max\{1,q\lambda\}$.
Starting from $\Lambda$, induction therefore gives scale
$\max\{1,\Lambda q^k\}$ for every integer $k\geq0$.
A finite $k$ satisfies $\Lambda q^k\leq1$, proving the assertion.
\end{proof}

Liao's inverse bound and the initial conditional-copy estimate supply
all hypotheses at one common scale. The strict comparison reduces
that scale to one.

\begin{lemma}[Diagonal inverse bound]\label{lem:diagonal:refined}
Suppose $\tau$ is a two-sided-admissible penalty on $G$.
Then $\mathsf D_{\tau^{[N]}}(c)$ and
$\mathsf C_{j,\tau^{[N]}}(\vartheta)$ hold for every $N\geq1$ and
$j\in\{0,\ldots,83\}$.
\end{lemma}
\begin{proof}
For $\lambda\geq1$, let $\mathsf H(\lambda)$ mean that
$\mathsf D_{(\tau/\lambda)^{[N]}}(B)$ and
$\mathsf C_{j,(\tau/\lambda)^{[N]}}(1)$ hold for every $N,j$.
Liao~\cite[Lemma~7]{l24}, with both penalties equal to $\tau^{[N]}$,
gives $\mathsf G_{\tau^{[N]}}(8)$ and hence
$\mathsf D_{\tau^{[N]}}(8)$.
Lemma~\ref{lem:branch:initial}, after division by $\eta$, gives
\[
M_{\tau^{[N]}}-2\tau^{[N]}(A\mid C)\leq8(x+s).
\]
These statements hold for every $N$ by Lemma~\ref{lem:penalty:chain}.
The constants in Definitions~\ref{def:diagonal:constants}
and~\ref{def:diagonal:branches} satisfy
\begin{equation}\label{eq:closed:initialization}
8\leq8B,\qquad8\leq8\alpha_j,\qquad8\leq8\beta_j.
\end{equation}
Dividing the initial inequalities by eight and applying the scaling
identity in Lemma~\ref{lem:penalty:chain} proves $\mathsf H(8)$.

Apply Lemma~\ref{lem:finite:scaling} to the family consisting, on
every chain power, of the diagonal bound with budget $B\dist[X;X]$
and the 84 cyclic bounds with budgets $\alpha_jx+\beta_js$.
Lemma~\ref{lem:penalty:chain} gives homogeneity.
Lemma~\ref{lem:diagonal:comparison} gives the common improvement factor
$\vartheta<1$, and the initialization gives $\Lambda=8$.
Thus $\mathsf H(1)$ holds. Applying
Lemma~\ref{lem:diagonal:comparison} once more gives the asserted
coefficients $c$ and $\vartheta$.
\end{proof}

One of the cyclic bounds has a short numerical consequence.

\begin{corollary}[Conditional cyclic estimate]\label{lem:diagonal:cyclic:target}
Suppose $\tau$ is two-sided-admissible. Every cyclic law on $G^N$ satisfies
\begin{equation}\label{eq:diagonal:cyclic:target}
M_{\tau^{[N]}}-2\tau^{[N]}(A\mid C)\leq4x+\frac{18}{5}s.
\end{equation}
\end{corollary}
\begin{proof}
Lemma~\ref{lem:diagonal:refined} gives the bound with coefficients
$\vartheta\alpha_{67}$ and $\vartheta\beta_{67}$.
Definition~\ref{def:finite:diagonal} gives $\alpha_{67}\leq4$ and
$\beta_{67}\leq18/5$. Since $\vartheta<1$ and $x,s\geq0$, the result follows.
\end{proof}

\subsection{Entropic and covering consequences}\label{sec:proofs:consequences}

Ruzsa distance to a fixed law is a two-sided-admissible penalty.
It gives the diagonal entropic consequence and the two-law consequence
in Appendix~\ref{app:two:law}.

\begin{lemma}[Reference-distance penalty]\label{lem:reference:admissible}
Suppose $R$ is a fixed $G$-valued random variable. Then
$\sigma_R(Z):=\dist[R;Z]$ is a two-sided-admissible penalty.
\end{lemma}
\begin{proof}
Suppose $Q$ is jointly distributed with $Z$, and take $R$ independent
of $(Z,Q)$.
\begin{samepage}
We have
\begin{align}
\sigma_R(Z\mid Q)-\sigma_R(Z)
&=\HH[R-Z\mid Q]-\HH[R-Z]
 +\tfrac12(\HH[Z]-\HH[Z\mid Q])\notag\\
&=\tfrac12\II[Z:Q]-\II[R-Z:Q].
\label{eq:sigma:conditioning}
\end{align}
The first step follows from the definition of $\sigma_R$.
The last step follows from the definition of mutual information.
\par
\end{samepage}
The Markov chain $Q\to Z\to R-Z$ gives
$0\leq\II[R-Z:Q]\leq\II[Z:Q]$ by data processing.
Substituting this bound into Eq.~\eqref{eq:sigma:conditioning} proves
two-sided conditioning. Translation invariance follows from
Definition~\ref{def:ruzsa}.
Lemma~\ref{lem:penalty:core} gives continuity and convolution
admissibility, completing the proof.
\end{proof}

Applying the reference penalty to two copies of one law gives the
diagonal entropic bound.

\begin{corollary}[Diagonal entropic bound]\label{cor:entropic}
Suppose $X$ is a $G$-valued random variable.
Then there exists a subspace $V\leq G$ such that
\[
    2\dist[X;\cU_V]
    \leq
    \frac{1257363397}{200000000}\dist[X;X].
\]
\end{corollary}

\begin{proof}
Set $\tau(Z):=\dist[X;Z]$, which is two-sided-admissible by
Lemma~\ref{lem:reference:admissible}.
Apply Lemma~\ref{lem:diagonal:refined} and choose a subspace $V$
attaining $M_\tau$. The assertion $\mathsf D_\tau(c)$ becomes
\[
2\dist[X;\cU_V]\leq(2+c)\dist[X;X].
\]
Since $2+c=1257363397/200000000$, this is the required bound.
\end{proof}

Projected KL divergence is convex, while its sum with entropy is concave.
The two-sided conditioning criterion then gives its regularity.

\begin{lemma}[Regularity of projected KL divergence]\label{lem:KL:conditioning}
Suppose $\mathcal C$ is a nonempty convex family of probability distributions
on a finite set $\Omega=\{\omega_1,\ldots,\omega_q\}$ and contains a
full-support distribution.
Define $J_{\mathcal C}(P):=\inf_{Q\in\mathcal C}\KL(P\,\|\,Q)$.
This functional is continuous. For every finite joint law of $X,D$ with $X$ taking values in $\Omega$,
\begin{equation}\label{eq:KL:projection:conditioning}
0\leq\E_D J_{\mathcal C}(\operatorname{Law}(X\mid D))
 -J_{\mathcal C}(\operatorname{Law}(X))\leq\II[X:D].
\end{equation}
Consequently, for every real constant $c$, the functional
$\rho(P):=J_{\mathcal C}(P)+\frac12\HH[P]-c$ satisfies
\begin{equation}\label{eq:KL:two:sided}
|\rho(X\mid D)-\rho(X)|\leq\frac12\II[X:D].
\end{equation}
\end{lemma}
\begin{proof}
The full-support member makes $J_{\mathcal C}$ finite everywhere.
Joint convexity of KL divergence and convexity of $\mathcal C$ imply
that $J_{\mathcal C}$ is convex: for a finite mixture
$P=\sum_{i=1}^m p_iP_i$, mix reference laws whose divergences approach
$J_{\mathcal C}(P_i)$ and apply the log-sum inequality.

Fix a full-support $Q_0\in\mathcal C$.
Replacing $Q\in\mathcal C$ by $(1-\epsilon)Q+\epsilon Q_0$ and letting
$\epsilon\downarrow0$ leaves the infimum unchanged.
Thus, if $\mathcal C^+$ is the family of full-support members, then
\begin{equation}\label{eq:KL:cross:entropy}
J_{\mathcal C}(P)+\HH[P]
=\inf_{Q\in\mathcal C^+}-\sum_{i=1}^qP(\omega_i)\log Q(\omega_i).
\end{equation}
The right side is an infimum of affine functions of $P$, hence is
concave. Jensen's inequality for the convex functional $J_{\mathcal C}$
gives the lower bound in Eq.~\eqref{eq:KL:projection:conditioning}.
Jensen's inequality for the concave functional $J_{\mathcal C}+\HH$
gives the upper bound.

For $\rho=J_{\mathcal C}+\HH/2-c$, the two functionals in
Lemma~\ref{lem:penalty:core} are
$\Phi_+=2(J_{\mathcal C}+\HH)-2c$ and $\Phi_-=2c-2J_{\mathcal C}$.
Both are concave. That lemma gives Eq.~\eqref{eq:KL:two:sided}
and continuity of $\rho$. Since entropy is continuous,
$J_{\mathcal C}=\rho-\HH/2+c$ is continuous as well.
\end{proof}

We use the averaged KL penalty of Liao~\cite{l24} in a single expression.
Its subgroup values record the densest coset of the reference set.

\begin{definition}[KL covering penalty]\label{def:KL:covering}
Suppose $A\subseteq G$ is nonempty. Define
$\mathcal C_A:=\{\cU_A+T:T\text{ is a law on }G\}$, using independent sums,
and
\[
\tau_A(P):=\inf_{Q\in\mathcal C_A}\KL(P\,\|\,Q)
+\frac12\HH[P]-\frac12\log|A|.
\]
For $V\leq G$, define $m_A(V):=\max_{t\in G}|A\cap(V+t)|$.
\end{definition}

The projected-KL lemma verifies the penalty hypotheses.
The subgroup calculation is Liao's dense-coset identity~\cite[Lemma~11]{l24}.

\begin{lemma}[KL penalty and subgroup values]\label{lem:KL:covering}
For every nonempty $A\subseteq G$, the penalty $\tau_A$ in
Definition~\ref{def:KL:covering} is two-sided-admissible and
$\tau_A(\cU_A)=0$. For every $V\leq G$,
\begin{equation}\label{eq:KL:subgroup:value}
2\tau_A(\cU_V)=\log\frac{|A||V|}{m_A(V)^2}.
\end{equation}
\end{lemma}
\begin{proof}
The family $\mathcal C_A$ is convex, contains $\cU_G$, and is invariant
under translations. Thus $\tau_A$ is translation invariant, and
Lemma~\ref{lem:KL:conditioning} gives two-sided conditioning.
Lemma~\ref{lem:penalty:core} now gives two-sided admissibility.
At $P=\cU_A$, the KL infimum is zero by nonnegativity and the choice
$T\equiv0$. The entropy terms cancel, so $\tau_A(\cU_A)=0$.

Fix $V\leq G$ and set $m:=m_A(V)>0$.
For $Q=\operatorname{Law}(\cU_A+T)$, averaging over $T$ gives
\begin{equation}\label{eq:KL:coset:mass}
Q(V)=\E_T\frac{|A\cap(V-T)|}{|A|}\leq\frac m{|A|}.
\end{equation}
\begin{samepage}
We have
\begin{equation}\label{eq:KL:subgroup:lower}
\KL(\cU_V\,\|\,Q)\geq-\log Q(V)\geq\log(|A|/m).
\end{equation}
The first step is Jensen's inequality on $V$, with the inequality
automatic if $Q$ vanishes at a point of $V$.
The last step follows from Eq.~\eqref{eq:KL:coset:mass}.
\par
\end{samepage}
Choose $t_0$ attaining $|A\cap(V+t_0)|=m$ and take $T$ uniform on
$V-t_0$. Then $Q(v)=m/(|A||V|)$ for every $v\in V$, so equality holds
in Eq.~\eqref{eq:KL:subgroup:lower}.
Consequently $J_{\mathcal C_A}(\cU_V)=\log(|A|/m)$.
Substitution into Definition~\ref{def:KL:covering} proves
Eq.~\eqref{eq:KL:subgroup:value}.
\end{proof}

The diagonal inverse bound supplies a dense coset through
Lemma~\ref{lem:KL:covering}; the dense-slice covering lemma then gives
the required cover.

\begin{theorem}[Main upper bound, formal version of
Theorem~\ref{thm:main}]\label{thm:main:formal}
Suppose $n$ is a positive integer, $A\subseteq\F_2^n$ is nonempty, and
$K\geq1$ satisfies $|A+A|\leq K|A|$.
Define $C_0:=1057363397/200000000$.
Then there is a subspace $H\leq\F_2^n$ with $|H|\leq|A|$ such that $A$ is
covered by at most $2K^{C_0}$ translates of $H$.
\end{theorem}
\begin{proof}
Use the penalty $\tau_A$ from Definition~\ref{def:KL:covering}.
Lemma~\ref{lem:KL:covering} makes it two-sided-admissible and gives
$\tau_A(\cU_A)=0$.
Hence Lemma~\ref{lem:diagonal:refined} applies.
Choose $V\leq G$ attaining $M_{\tau_A}$, and set
$d:=\dist[\cU_A;\cU_A]$ and $m:=m_A(V)$.
The diagonal inverse bound gives
\begin{equation}\label{eq:KL:bound}
M_{\tau_A}\leq cd.
\end{equation}
Since $\cU_A+\cU_A$ is supported on $A+A$,
\begin{samepage}
We have
\begin{equation}\label{eq:uniform:self:distance}
d=\HH[\cU_A+\cU_A]-\log|A|
\leq\log(|A+A|/|A|)\leq\log K.
\end{equation}
The first step is Definition~\ref{def:ruzsa} for the uniform law.
The second step bounds entropy by the logarithm of the support size.
The last step uses the doubling hypothesis.
\par
\end{samepage}

Since $m\leq\min\{|A|,|V|\}$, the logarithm
$\log(\min\{|A|,|V|\}/m)$ is nonnegative.
\begin{samepage}
We have
\begin{equation}\label{eq:KL:dense:budget}
\log\frac{\max\{|A|,|V|\}}m
\leq\log\frac{|A||V|}{m^2}
=M_{\tau_A}\leq cd\leq c\log K.
\end{equation}
The first step adds this nonnegative logarithm.
The second step is the subgroup identity in Lemma~\ref{lem:KL:covering}
at the minimizing $V$.
The third step uses Eq.~\eqref{eq:KL:bound}.
The last step uses Eq.~\eqref{eq:uniform:self:distance}.
\par
\end{samepage}
Thus a coset $V+t$ attaining $m$ satisfies
\begin{equation}\label{eq:dense:slice}
|A\cap(V+t)|\geq K^{-c}\max\{|A|,|V|\}.
\end{equation}
Apply Lemma~\ref{lem:slice:cover} with $R:=K^{c}$.
It gives a subspace $H$ with $|H|\leq|A|$ and a cover by at most
$2KR=2K^{1+c}=2K^{C_0}$ translates of $H$, as required.
\end{proof}

%% file: 90_llm_disclaimer.tex
\section*{GenAI Usage Disclosure}

Generative AI tools, including Codex, assisted with mathematical exploration,
proof development, and manuscript preparation.

%% file: _3_app.tex
\clearpage
\appendix
\input{32_two_law}
\clearpage
\input{33_coefficients}

%% file: 32_two_law.tex
\section{The two-law entropic bound}\label{app:two:law}

This appendix proves the two-law entropic bound in
Corollary~\ref{cor:entropic:general}. We first establish a general inverse
bound, then refine it using the main diagonal and cyclic estimates.

\subsection{Simultaneous comparison and scaling}

The finite construction in Appendix~\ref{app:coefficients} supplies the
coefficients for the simultaneous comparison.
Lemma~\ref{lem:finite:validity} gives the common lower bound for each cost.

\begin{lemma}[Conditional simultaneous improvements]\label{lem:outer:closing}
Suppose $\tau$ is two-sided-admissible and, for every $N\geq1$,
$\mathsf G_{\tau^{[N]}}(a)$, $\mathsf D_{\tau^{[N]}}(b)$, and
$\mathsf B_{j,\tau^{[N]}}(1)$ hold for $0\leq j\leq10$.
Then, for every $N\geq1$, $\mathsf G_{\tau^{[N]}}(a')$,
$\mathsf D_{\tau^{[N]}}(b')$, and
$\mathsf B_{j,\tau^{[N]}}(\theta)$ hold for all such $j$.
\end{lemma}
\begin{proof}
Fix $N$ and set $\rho:=\tau^{[N]}$.
Lemma~\ref{lem:penalty:chain} supplies the hypotheses on all further
chain powers of $\rho$.
Use the thirteen identities selected by
Definition~\ref{def:finite:general}; their existence is
Lemma~\ref{lem:finite:feasibility}.

For an inner diagram set $L:=\eta(M_\rho-2\rho(A\mid C))$.
For an outer diagram set $L:=\eta(M_\rho-\tauzero)$, using
Definition~\ref{def:outer:copies} in the general case and
Definition~\ref{def:outer:diagonal} in the diagonal case.
Lemma~\ref{lem:finite:validity} makes every generated row nonnegative.
In particular, general costs use the coefficient $a$.
Matching ensembles and complete weighted unions use the coefficient $b$.
Quotient costs use the eleven branch hypotheses directly.

The exact cone identities express each of
$\theta(u_jx+v_js)-L$, $\eta a'd-L$, and $\eta b'd-L$
as a nonnegative rational combination of these rows in its corresponding
diagram. The first family proves
$\mathsf B_{j,\rho}(\theta)$.
Dividing the last two inequalities by $\eta$ proves
$\mathsf G_\rho(a')$ and $\mathsf D_\rho(b')$.
The construction is uniform in $N$ and remains valid under further
finite conditioning.
\end{proof}

One scale parameter closes all the bounds on every chain power.

\begin{lemma}[Simultaneous inverse and branch bounds]\label{lem:penalty:improved}
Suppose $\tau$ is a two-sided-admissible penalty on $G$. Then, for every
$N\geq1$, $\mathsf G_{\tau^{[N]}}(a')$,
$\mathsf D_{\tau^{[N]}}(b')$, and
$\mathsf B_{j,\tau^{[N]}}(\theta)$ hold for $0\leq j\leq10$.
\end{lemma}
\begin{proof}
For $\lambda\geq1$, let $\mathsf H(\lambda)$ mean that all unprimed
hypotheses of Lemma~\ref{lem:outer:closing} hold for $\tau/\lambda$
on every chain power.
Liao~\cite[Lemma~7]{l24} gives both inverse bounds at coefficient $8$.
Lemma~\ref{lem:branch:initial} gives the initial cyclic budget
$8\eta(x+s)$. The constants and the grid in
Definition~\ref{def:finite:general} satisfy
\begin{align}
8&\leq4\min\{a,b\},\label{eq:scaling:initial:inverse}\\
8\eta&\leq4\min\{u_j,v_j\},\label{eq:scaling:initial:shapes}\\
\max\{a'/a,b'/b\}&<\theta<1.\label{eq:scaling:common:ratio}
\end{align}
The first and third assertions follow by rational arithmetic in
Definition~\ref{def:copy:coefficients}.
The second follows from $u_j,v_j\geq2/5$ and $\eta=2/11$.
Dividing the initial bounds by four proves $\mathsf H(4)$.

Apply Lemma~\ref{lem:finite:scaling} to the two inverse budgets
$a\dist[X;Y]$, $b\dist[X;X]$, and the eleven cyclic budgets
$u_jx+v_js$ on every chain power.
Lemma~\ref{lem:penalty:chain} gives homogeneity.
Lemma~\ref{lem:outer:closing} and
Eq.~\eqref{eq:scaling:common:ratio} give the common factor $\theta$.
Thus $\mathsf H(1)$ holds.
One final application of Lemma~\ref{lem:outer:closing} gives the result.
\end{proof}

\subsection{Refinement from the diagonal and cyclic bounds}

The main diagonal comparison supplies stronger inputs for a second
six-copy comparison. Define
\begin{equation}\label{eq:aux:coefficients}
a_0:=\frac{4793}{1000},\qquad
 a_1:=\frac{4792457}{1000000},\qquad
 q:=\frac{a_1}{a_0}=\frac{4792457}{4793000}<1.
\end{equation}
The finite construction in Definition~\ref{def:finite:auxiliary} selects
a diagonal coefficient tuple and one additional general identity.

\begin{lemma}[Cyclic bounds for an accepted tuple]\label{lem:aux:cyclic}
Suppose $S=((\widehat\alpha_j,\widehat\beta_j))_{j=0}^{83}$ passes
all tests in Definition~\ref{def:finite:diagonal}.
For every two-sided-admissible penalty $\tau$ and every $N\geq1$,
$\mathsf D_{\tau^{[N]}}(c)$ holds. Every cyclic law on $G^N$ satisfies
\begin{equation}\label{eq:aux:cyclic}
M_{\tau^{[N]}}-2\tau^{[N]}(A\mid C)
\leq\widehat\alpha_jx+\widehat\beta_js
\qquad(0\leq j\leq83).
\end{equation}
\end{lemma}
\begin{proof}
The proof of Lemma~\ref{lem:diagonal:comparison} uses only the grid,
source-pair tests, and 85 cone identities of
Definition~\ref{def:finite:diagonal}; it does not use the choice of the
first accepted tuple. It therefore applies to $S$.
Since $B,\widehat\alpha_j,\widehat\beta_j\geq1$, Liao's coefficient $8$
and Lemma~\ref{lem:branch:initial} initialize all hypotheses at scale
eight on every chain power. Lemma~\ref{lem:finite:scaling}, with the
common factor $\vartheta$, gives scale one exactly as in
Lemma~\ref{lem:diagonal:refined}.
One final comparison gives the diagonal coefficient $c$ and cyclic
coefficients $\vartheta(\widehat\alpha_j,\widehat\beta_j)$.
Since $x,s\geq0$ and $\vartheta<1$, these imply
Eq.~\eqref{eq:aux:cyclic}. Further finite conditioning is allowed.
\end{proof}

Only the general inverse hypothesis remains to be closed.

\begin{lemma}[Refined general inverse bound]\label{lem:aux:general}
Suppose $\tau$ is a two-sided-admissible penalty on $G$.
Then $\mathsf G_{\tau^{[N]}}(a_1)$ holds for every $N\geq1$.
\end{lemma}
\begin{proof}
Let $S^\dagger$ be the tuple selected by
Definition~\ref{def:finite:auxiliary}; Lemma~\ref{lem:finite:auxiliary}
proves that this construction is defined.
Lemma~\ref{lem:aux:cyclic} supplies its cyclic bounds and the diagonal
coefficient $c$ for every two-sided-admissible penalty on all chain powers.

First let $\rho$ be two-sided-admissible and suppose
$\mathsf G_{\rho^{[N]}}(a_0)$ holds for every $N\geq1$.
Fix $N$, put $\sigma:=\rho^{[N]}$, and use three independent copies
of each input law $X,Y$. Set $L:=M_\sigma-\sigma(X)-\sigma(Y)$ and
$d:=\dist[X;Y]$. The selected cone identity has the form
\begin{equation}\label{eq:aux:identity}
a_1d-L=\sum_{r=1}^{m}w_rJ_r,\qquad
w_r>0,\qquad 1\leq m\leq512.
\end{equation}
Here $m$ is the number of rows in the selected identity.
General costs use $a_0$, diagonal ensemble costs use $c$, and quotient
costs use the pairs in $S^\dagger$.
Lemma~\ref{lem:finite:validity} makes every row nonnegative.
In particular, a weighted union of $X$ and $Y$ with equal weights gives
\begin{equation}\label{eq:aux:mixture}
\frac c4\bigl(\dist[X;X]+2\dist[X;Y]+\dist[Y;Y]\bigr)-L\geq0.
\end{equation}
Both self-distance terms and the full cross term are retained.
All conditional inverse costs use independent posterior labels and
replicas. Lemma~\ref{lem:penalty:chain} supplies the premises on every
further chain power of $\sigma$.
Thus Eq.~\eqref{eq:aux:identity} proves
$\mathsf G_{\rho^{[N]}}(a_1)$ for every $N$, including when $d=0$.

For $\lambda\geq1$, let $\mathsf H(\lambda)$ mean
$\mathsf G_{\tau^{[N]}}(\lambda a_0)$ for every $N\geq1$.
Liao~\cite[Lemma~7]{l24} gives $\mathsf H(8/a_0)$.
Apply the preceding comparison to $\rho=\tau/\lambda$ and use
Lemma~\ref{lem:penalty:chain} to obtain
\begin{equation}\label{eq:aux:scaling}
\mathsf H(\lambda)\Longrightarrow
\mathsf H(\max\{1,q\lambda\}).
\end{equation}
The source bounds apply to $\rho$ because they are universal over the
penalty class. Lemma~\ref{lem:finite:scaling} therefore gives
$\mathsf H(1)$. Equivalently, rational arithmetic gives the explicit bound
\[
\frac{8000}{4793}\left(\frac{4792457}{4793000}\right)^{5000}<1.
\]
One final application of the comparison proves the result.
\end{proof}

\subsection{Entropic consequence}

For two input laws, the average of their reference penalties permits the
same subspace to control both distances.

\begin{corollary}[Two-law entropic bound]\label{cor:entropic:general}
Suppose $X,Y$ are $G$-valued random variables.
Then there exists a subspace $V\leq G$ such that
\[
\dist[X;\cU_V]+\dist[Y;\cU_V]\leq\frac{7792457}{1000000}\dist[X;Y].
\]
\end{corollary}
\begin{proof}
Set $d:=\dist[X;Y]$ and
$\tau(Z):=\tfrac12(\dist[X;Z]+\dist[Y;Z])$.
Lemma~\ref{lem:reference:admissible} and averaging imply that $\tau$
is two-sided-admissible.
Lemma~\ref{lem:aux:general} gives $\mathsf G_\tau(a_1)$.
A subspace $V$ attaining $M_\tau$ therefore satisfies
\begin{equation}\label{eq:reference:averaged}
\dist[X;\cU_V]+\dist[Y;\cU_V]
\leq(1+a_1)d+\tfrac12(\dist[X;X]+\dist[Y;Y]).
\end{equation}

We use the entropic Ruzsa triangle inequality. For independent $P,Q,R$,
\begin{samepage}
We have
\begin{align*}
\HH[P+Q]+\HH[Q+R]
&\geq\HH[P+Q,Q+R]\\
&=\HH[P+R,Q+R]\\
&=\HH[P+R]+\HH[Q+R\mid P+R]\\
&\geq\HH[P+R]+\HH[Q+R\mid P+R,R]\\
&=\HH[P+R]+\HH[Q].
\end{align*}
The first step follows from subadditivity.
The second step follows by adding the two coordinates in characteristic two.
The third step follows from the chain rule.
The fourth step follows because further conditioning decreases entropy.
The last step follows from independence of $Q$ from $(P,R)$.
\par
\end{samepage}
Subtracting the marginal half-entropies gives
$\dist[P;R]\leq\dist[P;Q]+\dist[Q;R]$.
Applying this with input laws $(X,Y,X)$ and $(Y,X,Y)$ gives, respectively,
\begin{align}
\dist[X;X]&\leq2d,\label{eq:reference:diagonal:X}\\
\dist[Y;Y]&\leq2d.\label{eq:reference:diagonal:Y}
\end{align}
Adding Eqs.~\eqref{eq:reference:diagonal:X} and~\eqref{eq:reference:diagonal:Y}
yields
\begin{equation}\label{eq:reference:diagonal:sum}
\tfrac12(\dist[X;X]+\dist[Y;Y])\leq2d.
\end{equation}
\begin{samepage}
We have
\begin{align*}
\dist[X;\cU_V]+\dist[Y;\cU_V]
&\leq(1+a_1)d+\tfrac12(\dist[X;X]+\dist[Y;Y])\\
&\leq(3+a_1)d\\
&=\tfrac{7792457}{1000000}d.
\end{align*}
The first step follows from Eq.~\eqref{eq:reference:averaged}.
The second step follows from Eq.~\eqref{eq:reference:diagonal:sum}.
The last step follows from Eq.~\eqref{eq:aux:coefficients}.
\par
\end{samepage}
This is the asserted bound.
\end{proof}

%% file: 33_coefficients.tex
\section{Finite rational constructions}\label{app:coefficients}

This appendix specifies the finite coefficient constructions used in
the comparison lemmas. The feasibility assertions for these constructions
are verified using rational arithmetic.

\subsection{Formal coordinates and admissible rows}

Binary subspaces and flags give a finite coordinate system for each copy
diagram. Equalities from independence and symmetry are imposed before
testing a coefficient identity.

\begin{definition}[Formal copy coordinates]\label{def:finite:coordinates}
For an inner diagram use 32 fresh copies and, for unconditional
entropies, the additional coordinate $C$, as in
Definition~\ref{def:cyclic:copies}. For a diagonal outer diagram use
36 copies as in Definition~\ref{def:outer:diagonal}. For a general outer
diagram use eight copies of each input as in
Definition~\ref{def:outer:copies}. Unused coordinates are allowed.
Encode a binary vector $(v_0,\ldots,v_{n-1})$ by
$\sum_{i=0}^{n-1}2^iv_i$, and encode a subspace by its reduced row-echelon
basis, using the largest occupied bit as pivot and listing pivots in
decreasing order. A flag is an encoded pair $W\subset U$ with
$\dim(U/W)=1$.

Take one formal entropy coordinate for each subspace and one formal
penalty coordinate for each flag. Use $\mathscr E,\mathscr U,\mathscr R$
in the inner setting and $\mathcal E,\mathcal R$ in the outer setting.
Adjoin $x,s,L$ in the inner setting and $d,L$ in the outer setting.
Quotient this rational vector space by the following equalities:
\begin{enumerate}
\item Empty entropy coordinates vanish. Conditional-copy and outer
entropy coordinates vanish on one input coordinate and add on disjoint
supports. Row operations and copy permutations identify entropy
coordinates. General outer permutations preserve the two input classes
or interchange them globally.
\item In the inner setting, $\mathscr U(C)=s-x$,
$\mathscr U(Z_i)=\mathscr U(C+Z_i)=s/2$, and
$\mathscr U(C,Z_W)=s-x+\mathscr E(W)$.
Unconditional entropy coordinates are also identified under simultaneous
reflection $Z_i\mapsto Z_i+C$; their permutations fix $C$.
\item Penalty coordinates are identified when their flags agree after
copy permutations and removal of independent irrelevant conditioning.
An output may change by a vector in its conditioning space.
A single unconditioned input has penalty coordinate zero.
In the general outer setting, global interchange is allowed for an
averaged coordinate, but an ensemble matching within one orientation
must preserve both input classes separately.
\item In the outer setting, $d$ is the entropy coordinate of
$\operatorname{span}(Z_1+Z_2)$ in the diagonal case and of
$\operatorname{span}(X_1+Y_1)$ in the general case.
\end{enumerate}
Independent components are the connected components of the supports
of the reduced basis rows. To remove irrelevant conditioning, reduce
the output modulo $W$, retain the components of $W$ meeting its support,
and discard the others. Impose all these equalities for every permitted
subspace, flag, and permutation. Rational Gaussian elimination determines
the resulting quotient.
\end{definition}

In the next definition, $E,R$ denote the conditional-copy or outer
functionals, and $h_f:=E(U_f)-E(W_f)$. Inner unconditional entropy
rows use $\mathscr U$ explicitly.

\begin{definition}[Finite row families]\label{def:finite:rows}
Fix a positive rational scale $\epsilon$, licensed cyclic pairs $(a,b)$,
and licensed inverse coefficients. All subspaces and flags below range
over the corresponding diagram in Definition~\ref{def:finite:coordinates}.
\begin{enumerate}
\item Entropy rows are $E(P)$ and
$F(P)+F(Q)-F(P\cap Q)-F(P+Q)$, where $F=E$ or, in the inner setting,
$F=\mathscr U$. Also allow
$\sum_{i=1}^k\mathscr U(P_i)-\mathscr U(\sum_{i=1}^k P_i)$ when
$1\leq k\leq33$ and $\sum_{i=1}^k \dim(P_i)=\dim(\sum_{i=1}^k P_i)$.
The inner rows additionally include $x,s$ and
$\mathscr U(P)-\mathscr E(P^\circ)$ when $C\notin P$.
Here $P^\circ$ is the projection deleting $C$.
\item If $f$ is obtained by further conditioning $g$, allow
$R_g-R_f$ and $R_f-R_g+2(h_g-h_f)$.
If $f=g*h$ on disjoint supports, allow
$R_g+2(h_f-h_g)-R_f$ and $R_f-R_g$.
For every output representative $v\in U_f\setminus W_f$, allow
$2E(\operatorname{span}(v))-R_f$.
\item For disjoint flags $f=(W\subset U)$ and $g=(W'\subset U')$,
choose output representatives $v,v'$ and define
\begin{align}
D(f,g)&:=E(W+W'+\operatorname{span}(v+v'))
 -\tfrac12[E(U)+E(U')+E(W)+E(W')],
\label{eq:finite:distance}\\
Q_k(f,g)&:=kD(f,g)+\tfrac12(R_f+R_g-h_f-h_g).
\label{eq:finite:cost}
\end{align}
Allow $\epsilon Q_k(f,g)-L$ for a licensed general inverse coefficient,
or for a licensed diagonal coefficient when the posterior ensembles
match by a copy permutation after irrelevant conditioning is removed.
For a diagonal coefficient also allow
\begin{equation}\label{eq:finite:mixture}
\epsilon[p^2Q_k(f,f')+2p(1-p)Q_k(f,g')+(1-p)^2Q_k(g,g')]-L,
\end{equation}
where $p=r/t$ and $1\leq r<t\leq1000$.
Primes denote independent replicas. Each component must fit in the
diagram, and the self families must match the corresponding cross
families within each orientation.
\item For every cyclic quotient in
Definition~\ref{def:inner:quotient} or~\ref{def:outer:quotient}, with
selected posterior $f$ and quantities $x_f,s_f$, allow
\begin{equation}\label{eq:finite:branch}
ax_f+bs_f+\epsilon(R_f-h_f)-L
\end{equation}
for each licensed pair $(a,b)$.
\end{enumerate}
Also allow a common-function entropy row plus the block-subadditivity
rows for its two positive terms as one compound row. Thus a reflection
row applies the common-function inequality to $P,Q$, decomposes the
positive entropy terms into blocks, and normalizes the result.
Order rows by type, then by their integer encodings, and then by rational
parameters in increasing order. Duplicate vectors may be retained.
\end{definition}

All these row rules have a common probabilistic interpretation.

\begin{lemma}[Validity of the generated rows]\label{lem:finite:validity}
Suppose the inverse and cyclic hypotheses licensing
Definition~\ref{def:finite:rows} hold on every chain power of a
two-sided-admissible penalty. Substitute
$L=\epsilon(M_\tau-2\tau(A\mid C))$ in an inner diagram and
$L=\epsilon(M_\tau-\tauzero)$ in an outer diagram.
Then every generated row is nonnegative.
\end{lemma}
\begin{proof}
The formal equalities follow from Lemmas~\ref{lem:copy:normalization},
\ref{lem:copy:rules}, and~\ref{lem:outer:rules}.
Lemma~\ref{lem:ensemble:flags} gives the flag identifications.
The common-function rows follow from
Lemma~\ref{lem:entropy:common:function}; the rank terms cancel.
Block subadditivity has the same rank cancellation.
Lemma~\ref{lem:copy:normalization} gives the remaining inner entropy rows.

For $E(P)\geq0$, choose pivot coordinates for a basis of $P$ and
condition on all other inputs. The pivot variables are independent,
so the conditional tuple entropy is the sum of their entropies.
Removing this conditioning can only increase entropy. Each pivot has
averaged entropy $t$, also in the two orientations of the general case.
Subtracting $\dim(P)t$ proves the assertion.

The penalty rows follow from Lemmas~\ref{lem:copy:rules}
and~\ref{lem:outer:rules}. For reverse convolution, apply reverse
conditioning to the sum and then condition on the second summand.
The chain rule identifies $D(f,g)$ with the independently conditioned
Ruzsa distance. Set $T_0:=2\tau(A\mid C)$ for an inner diagram and
$T_0:=\tauzero$ for an outer diagram. The state normalizations give
\begin{equation}\label{eq:finite:cost:interpretation}
Q_k(f,g)=k\dist_{\mathrm{sym}}[f;g]
+\tau_{\mathrm{sym}}(f)+\tau_{\mathrm{sym}}(g)-T_0.
\end{equation}
A licensed general inverse bound applies at each pair of posterior
labels and then averages. For matching ensembles,
Lemma~\ref{lem:inverse:ensemble} gives the same bound with a licensed
diagonal coefficient. For a weighted union, its two independent labels
give weights $p^2,2p(1-p),(1-p)^2$.
The same lemma therefore bounds the complete mixture cost.
Multiplication by $\epsilon$ and subtraction of $L$ prove all inverse rows.
In the inner case these applications are first made at fixed $C$ and
then averaged. Finally Lemma~\ref{lem:cyclic:substitution}, with
coefficients $a/\epsilon,b/\epsilon$, proves the quotient rows.
The arguments remain valid after further conditioning and averaging.
\end{proof}

\subsection{Coefficient selection by rational linear systems}

For fixed coefficients, the remaining unknowns are nonnegative row
weights. The following rule specifies the finite feasibility test.

\begin{definition}[Exact cone test]\label{def:finite:cone}
For a finite row family $\mathcal R$ and target vector $T$, enumerate
subsets of at most $512$ rows, first by cardinality and then
lexicographically. Form the column matrix $A$ of each subset in the
quotient of Definition~\ref{def:finite:coordinates}.
Discard linearly dependent columns. Solve $Aw=T$ by rational Gaussian
elimination and accept if a solution exists and $w\geq0$.
The test succeeds if some subset is accepted.
The first accepted subset and its unique solution specify the chosen
identity; delete zero weights. The coefficient of $L$ ensures that
the weights on inverse and quotient costs sum to one.
\end{definition}

Only a bounded grid of coefficient tuples is needed. The source rows
are generated from the convex hull of the same tuple being tested.

\begin{definition}[Diagonal coefficient construction]\label{def:finite:diagonal}
Use $B,c,B_0,\vartheta$ from Definition~\ref{def:diagonal:constants}.
Enumerate tuples $S=((\alpha_j,\beta_j))_{j=0}^{83}$ satisfying
\begin{align}
&10^{10}\alpha_j,10^{10}\beta_j\in\mathbb Z,\quad
1\leq\alpha_j\leq6,\quad1\leq\beta_j\leq250,
\label{eq:finite:diagonal:grid}\\
&\alpha_0<\cdots<\alpha_{83},\qquad
\alpha_{67}\leq4,\qquad\beta_{67}\leq18/5.
\label{eq:finite:diagonal:order}
\end{align}
Order tuples lexicographically by
$(\alpha_0,\beta_0,\ldots,\alpha_{83},\beta_{83})$.
License a source pair if
\begin{equation}\label{eq:finite:source:grid}
(a,b)\in10^{-10}\mathbb Z^2\cap[1,250]^2,\qquad
\sum_{j=0}^{83}\lambda_j=1,\quad
\sum_{j=0}^{83}\lambda_j\alpha_j\leq a,\quad
\sum_{j=0}^{83}\lambda_j\beta_j\leq b
\end{equation}
for some $\lambda_j\geq0$. Test the latter rational linear system by
enumerating its basic solutions.
Only diagonal inverse coefficients are licensed: in an inner diagram
they range over $10^{-9}\mathbb Z\cap[B,B_0]$, and in the outer diagram
the coefficient is $B$. Set $\epsilon=1$.

Accept $S$ if the cone test succeeds for every target
\begin{align}
T_j&:=\vartheta\alpha_jx+\vartheta\beta_js-L
\quad(0\leq j\leq83),\label{eq:finite:diagonal:inner}\\
T_{\mathrm d}&:=cd-L.\label{eq:finite:diagonal:outer}
\end{align}
The first accepted tuple defines the coefficients in
Definition~\ref{def:diagonal:branches}; its cone tests define the 85 identities.
\end{definition}

For example, an unconditioned input has $R_f=h_f=0$.
Two such outer states give $D(f,g)=d$, so their direct row is $Bd-L$.
A mixture with $p=1/10$ has coefficients $1/100,9/50,81/100$.
These examples fix the cost normalization.

The separate two-law comparison uses the same construction with a
smaller coefficient tuple and scale $\eta$.

\begin{definition}[Two-law coefficient construction]\label{def:finite:general}
Use $\eta,a,b,a',b',\theta$ from Definition~\ref{def:copy:coefficients}.
Enumerate tuples $((u_j,v_j))_{j=0}^{10}$ with
\begin{equation}\label{eq:finite:general:grid}
10^{12}u_j,10^{12}v_j\in\mathbb Z,\qquad
2/5\leq u_j,v_j\leq13/10
\end{equation}
in lexicographic order. Set $q_j:=\theta$ for every $j$.
License these eleven cyclic pairs, the general inverse coefficient $a$,
and the diagonal inverse coefficient $b$. Set $\epsilon=\eta$.
Accept the tuple if the cone test succeeds for all thirteen targets
\begin{align}
T_j&:=\theta(u_jx+v_js)-L\quad(0\leq j\leq10),
\label{eq:finite:general:inner}\\
T_{\mathrm g}&:=\eta a'd-L
\quad\text{in the general outer diagram},
\label{eq:finite:general:outer}\\
T_{\mathrm d}&:=\eta b'd-L
\quad\text{in the diagonal outer diagram}.
\label{eq:finite:general:diagonal}
\end{align}
The first accepted tuple defines the coefficients in
Definition~\ref{def:branch:hypotheses}; its cone tests define the thirteen identities.
\end{definition}

All grids and row families are finite.
The following is the finite computational assertion used by the proof.

\begin{lemma}[Rational feasibility]\label{lem:finite:feasibility}
The families in Definitions~\ref{def:finite:diagonal}
and~\ref{def:finite:general} are nonempty.
Their selected coefficient tuples and identities are therefore defined.
\end{lemma}
\begin{proof}
Exact rational feasibility checks give comparison tuples in both printed
coefficient grids. For the diagonal tuple, all $457$ source pairs pass
the convex feasibility test, the inverse caps lie on the prescribed
$10^{-9}$ grid, and every mixture denominator is at most $1000$.
Expansion and normalization give inner outputs $(\alpha_j^*,\beta_j^*)$
and an outer output $\gamma$ satisfying
\begin{align}
&0<\alpha_j^*\leq\vartheta\alpha_j,\qquad
0<\beta_j^*\leq\vartheta\beta_j,
\label{eq:finite:verified:inner}\\
&\frac{4286816984787}{10^{12}}\leq\gamma
<\frac{4286816984788}{10^{12}}<c.
\label{eq:finite:verified:outer}
\end{align}
Add $(\vartheta\alpha_j-\alpha_j^*)x$ and
$(\vartheta\beta_j-\beta_j^*)s$ to each inner identity and
$(c-\gamma)d$ to the outer identity.
These nonnegative entropy or information rows give exactly the targets
of Definition~\ref{def:finite:diagonal}.
The largest resulting support has $454$ rows.

For the two-law tuple, the inner output ratios are at most
$499999779583/500000000000<\theta$, the outer coefficients are
strictly below $\eta a'$ and $\eta b'$, and the mixture denominators are
at most $20$. Adding the corresponding nonnegative $x,s,d$ rows gives
the targets of Definition~\ref{def:finite:general}; the largest support
has $222$ rows. Compound entropy rows are expanded by
Definition~\ref{def:finite:rows} before checking their coefficients.

For a dependent feasible support, take a nonzero rational dependence
among its columns. Subtract a multiple from the nonnegative weights
until one weight is zero and none is negative. Repeat until the
columns are independent. This preserves the target and never increases
support size, so all these identities pass the $512$-row cone test.

The feasibility assertions can be reproduced by the finite enumeration
of the printed grids and cone tests. This enumeration has no asserted
efficiency bound. The first accepted tuples are defined by the selection
rules; they need not coincide with the comparison tuples used to
establish nonemptiness.
\end{proof}

\subsection{A further two-law comparison}

The main diagonal construction can also supply the source pairs for
a general comparison using only three copies of each input law.

\begin{definition}[Refined two-law construction]\label{def:finite:auxiliary}
Use $a_0,a_1$ from Eq.~\eqref{eq:aux:coefficients} and $c$ from
Definition~\ref{def:diagonal:constants}.
Enumerate the tuples accepted by Definition~\ref{def:finite:diagonal}
in the same order. For each tuple $S$, form the general outer rows of
Definition~\ref{def:finite:rows}, restricted to three copies of each
input law, with $\epsilon=1$. License the general inverse coefficient
$a_0$, the diagonal inverse coefficient $c$, and the 84 cyclic pairs
in $S$. Apply the cone test of Definition~\ref{def:finite:cone} to
$a_1d-L$. The first tuple for which this additional test succeeds is
$S^\dagger$, and the cone test selects its general identity.
This tuple need not be the first tuple in
Definition~\ref{def:finite:diagonal}.
\end{definition}

A single rational witness proves nonemptiness of this additional test.

\begin{lemma}[Feasibility of the refined comparison]\label{lem:finite:auxiliary}
The construction in Definition~\ref{def:finite:auxiliary} is defined.
\end{lemma}
\begin{proof}
Use the diagonal comparison tuple from the rational verification in
Lemma~\ref{lem:finite:feasibility}.
For this same tuple, exact row expansion gives an identity
\begin{equation}\label{eq:aux:witness}
\gamma d-L=\sum_{r=1}^{69}w_rJ_r,\qquad w_r>0,
\end{equation}
where
\begin{equation}\label{eq:aux:gamma}
\frac{4792456014399348}{10^{15}}<\gamma
<\frac{4792456014399350}{10^{15}}<a_1<a_0.
\end{equation}
The rows consist of ten common-function entropy rows, ten forward
conditioning rows, fourteen reverse conditioning rows, eleven upper
convolution rows, six general inverse costs, one complete diagonal
ensemble mixture, and seventeen cyclic quotient costs.
The weights on the 24 inverse and quotient costs sum to one.
All auxiliary entropy and penalty coefficients cancel exactly.

The witness encodes its six positions in the order
$(X_1,X_2,X_3,Y_1,Y_2,Y_3)$; these embed in the general outer diagram.
Every conditioning inclusion and disjoint-support requirement is
checked on the literal binary spaces. The entropy and penalty
normalizations use the equalities in
Definition~\ref{def:finite:coordinates}, including global interchange
only for orientation-averaged quantities.
The fifteen cyclic source pairs are entries of the same 84-pair
comparison tuple used by the diagonal verification.
Their literal rational values, all 69 rows and weights, and an exact
verifier are supplied in \texttt{auxiliary\_4\_792457.zip}.
The package also records the full source tuple and the checksum of
its diagonal proof archive.

Add the nonnegative entropy row $(a_1-\gamma)d$ to
Eq.~\eqref{eq:aux:witness}. This gives the target $a_1d-L$ with at
most 70 rows. The dependence elimination in the proof of
Lemma~\ref{lem:finite:feasibility} preserves this target and does not
increase the support. Thus the 512-row cone test accepts this tuple,
proving nonemptiness. The selected tuple and identity may differ from
this witness.
\end{proof}